\documentclass[a4paper,10pt, reqno]{amsart}
\usepackage[backref]{hyperref}
\usepackage{a4wide}
\usepackage{amsthm}
\usepackage{amssymb}
\usepackage{mathrsfs}
 \newtheorem{theorem}{Theorem}[section]
 
 \newtheorem{lemma}[theorem]{Lemma}

\begin{document}

\title{Elliptic curves with square-free $\Delta$}

\author{Stephan Baier}
\address{Stephan Baier, Tata Institute of Fundamental Research, School of Mathematics, 1 Dr. Homi Bhaba Road, Colaba, Mumbai 400005, India}

\email{sbaier@math.tifr.res.in}

\subjclass[2000]{11L07,11D45}

\maketitle

{\bf Abstract:} Under the Riemann Hypothesis for Dirichlet L-functions, we improve on the error term in a smoothed version of an estimate for the density of elliptic curves with square-free $\Delta=D/16$, where $D$ is the discriminant, by T.D. Browning and the author \cite{BB}. To achieve this improvement, we elaborate on our methods for counting weighted solutions of inhomogeneous cubic congruences to power-ful moduli. The novelty lies in going a step further in the explicit evaluation of complete exponential sums and saving a factor by averaging over the moduli.\\ 

{\bf Keywords:} discriminants, elliptic curves, square-freeness, exponential sums, cubic congruences

\section{Main result}
Let $E$ be an elliptic curve over $\mathbb{Q}$, given in Weierstrass form
$$
E=E_{A,B}: \quad y^2=x^3+Ax+B
$$ 
for $A,B\in \mathbb{Z}$ with discriminant $-16(4A^3+27B^2)=-16\Delta_{A,B}$, say. Throughout the sequel, we shall implicitly assume that $A$ and $B$ are such that
$\Delta_{A,B}\not=0$, which is required for $E_{A,B}$ to be elliptic.

It is a very interesting question whether $\Delta_{A,B}$ is prime infinitely 
often. This is presently unknown. However, it is possible to estimate the density of elliptic curves with square-free $\Delta_{A,B}$. The exponential height of $E_{A,B}$ is defined as
$$
H\left(E_{A,B}\right)=\max\left\{|A|^{1/4},|B|^{1/6}\right\}.
$$
In \cite[Theorem 3]{BB}, we proved the following result, building on our work about inhomogeneous cubic congruences.

\begin{theorem} \label{density1}
For $	q\in \mathbb{N}$ let $\sigma(q):=\sharp\{\alpha,\beta \bmod{q} \ : \ \Delta_{\alpha,\beta}\equiv 0 \bmod{q}\}$. Let $\mu(n)$ be the M\"obius function, with the convention that 
$\mu(0)=0$ and $\mu(-n)=\mu(n)$. Then for any $\varepsilon>0$, we have
\begin{equation} \label{timi}
\sum\limits_{\substack{(A,B)\in \mathbb{Z}^2\\ H(E_{A,B})\le X}} \mu^2\left(\Delta_{A,B}\right)=4X^{10} \prod\limits_{p} \left(1-\frac{\sigma\left(p^2\right)}{p^4}\right) + O\left(X^{7+\varepsilon}\right).
\end{equation}
\end{theorem}

Here we present a smoothed version of this result with improved error term, where we assume the Riemann Hypothesis for Dirichlet $L$-functions. 

\begin{theorem} \label{density2}  Let $\Gamma : \mathbb{R} \rightarrow \mathbb{R}^+$ be a Schwartz class function and $\hat\Gamma$ its Fourier transform. Suppose that 
$\hat\Gamma(0)=1$. Assume that the Riemann Hypothesis holds for all Dirichlet $L$-functions. Then for any $\varepsilon>0$, we have
\begin{equation} \label{mainresult}
\sum\limits_{\substack{(A,B)\in \mathbb{Z}^2}} \Gamma\left(\frac{A}{X^4}\right)\Gamma\left(\frac{B}{X^6}\right) \mu^2\left(\Delta_{A,B}\right)=X^{10} \cdot 
\frac{1}{3}\cdot  \prod\limits_{p>3}  \left(1-\frac{2p-1}{p^3}\right) + O\left(X^{7-5/27+\varepsilon}\right).
\end{equation}
\end{theorem}

Thereby, we have also computed the Euler product in Theorem \ref{density1} explicitly. We note that the factor 4 in \eqref{timi}, which comes from counting positive {\it and} non-positive $A$'s and $B$'s, is not present in \eqref{mainresult} due to the fact that we work with smooth weights satisfying $\hat\Gamma(0)=1$.
 
Our method elaborates on that in \cite{BB}.  Rather than using our bounds for the number of solutions of inhomogeneous cubic congruences obtained in \cite{BB} directly, we here redo our treatment of them, specified for congruences of the form $4A^3+27B^2 \equiv 0 \bmod{k^2}$, and go a step further in the evaluation of exponential sums. Let us briefly describe how we proceed. 

First, we detect the square-freeness of $\Delta_{A,B}$ using the M\"obius function, leading to congruences of the form $4A^3+27B^2 \equiv 0 \bmod{k^2}$. To count their solutions, we apply the Poisson summation formula twice. This leads to complete cubic exponential sums to square moduli which we evaluate explicitly. Roughly, we are left to terms of the form
$$
e\left(-\frac{\overline{m}^2n^3}{k^2}\right),
$$
where $m$, $n$ are variables.  Now we flip the Kloosterman fraction by means of the identity
$$
e\left(-\frac{\overline{m}^2n^3}{k^2}\right)=e\left(\frac{\overline{k}^2n^3}{m^2}\right) e\left(-\frac{n^3}{m^2k^2}\right).
$$
The second exponential term on the right-hand side is slowly oscillating if considered as a function of $n$. We sum up over $n$ and use Poisson summation again, but now for modulus $m^2$. This leads to cubic exponential integrals and again to complete cubic exponential sums. So far, our method agrees with that in \cite{BB}. The novelty comes in the next steps. We evaluate the said complete cubic exponential sums for modulus $m^2$ explicitly, in contrast to our work in \cite{BB}, where we just bounded them. Moreover, we work with an asymptotic estimate for the said cubic exponential integrals instead of  upper bounds. Then we average over the moduli, which essentially leads to linear exponential sums with M\"obius function of the form
$$
\sum\limits_{k\le x} \mu(k)\cdot e(w k),
$$
where $w$ is a real number. It is this averaging which gives us an extra saving. 

In our work, we shall follow the usual convention that $\varepsilon$ can change from  line to line if no confusions are anticipated. \\

{\bf Acknowledgement.} The author wishes to thank the Tata Institute of Fundamental Research in Mumbai (India) for its warm hospitality, excellent working conditions and financial support by an ISF-UGC grant.  

\section{Reduction to inhomogeneous cubic congruences}
We start as in \cite{BB}. Using
$$
\mu^2(n)=\sum\limits_{k^2|n} \mu(k),
$$
we write 
\begin{equation} \label{ksums}
\sum\limits_{\substack{(A,B)\in \mathbb{Z}^2}} \Gamma\left(\frac{A}{X^4}\right)\Gamma\left(\frac{B}{X^6}\right) \mu^2\left(\Delta_{A,B}\right)=
\sum\limits_{k=1}^\infty \mu(k) S(X;k^2)=:S(X), 
\end{equation}
where 
$$
S(X;k^2):=\sum\limits_{\substack{(A,B)\in \mathbb{Z}^2\\ 4A^3+27B^2 \equiv 0 \bmod{k^2} }} \Gamma\left(\frac{A}{X^4}\right)\Gamma\left(\frac{B}{X^6}\right).
$$
We split the sum over $k$ into two parts, according to the size of $k$. Let $1\le \xi \le X^{100}$ be a parameter, to be fixed later. We define
\begin{equation} \label{splitting}
S(X)=S_1(X)+S_2(X),
\end{equation}
where 
\begin{equation} \label{S1}
S_1(X):=\sum\limits_{k\le \xi} \mu(k) S(X;k^2) 
\end{equation}
and
\begin{equation} \label{S2}
S_2(X):= \sum\limits_{k> \xi} \mu(k) S(X;k^2).
\end{equation}
Following \cite[Eq. (2.3)]{BB}, using an argument of Estermann \cite{Es}, we have 
$$
S_2(X)\ll \sum\limits_{|A|\le X^{4+\varepsilon}}  \sum\limits_{0< |m|\ll X^{12+\varepsilon}\xi^{-2}} \mathop{\sum\limits_{|B|\le X^{6+\varepsilon}} \sum\limits_{k\in \mathbb{N}}}_{4A^3=mk^2-27B^2} 1\ll X^{16+\varepsilon}\xi^{-2}.
$$

\section{Application of Poisson summation I}
We now transform the term $S_1(X)$. First, we rewrite $S_1(X;k^2)$, where we assume that $k$ is square-free. As in \cite{BB}, we write $k=k_2k_3k'$, where $k_2=(k,2)$ and 
$k_3=(k,3)$ and $k'$ is coprime to 6. It readily
follows that $k_2|B$ and $k_3|A$ in the summand. Making the change of variables $A=k_3A'$ and $B = k_2B'$ we deduce that
$$
S_1(X)=\sum\limits_{k_2|2} \sum\limits_{k_3|3} \mu(k_2)\mu(k_3) \sum\limits_{\substack{k'\le\xi/(k_2k_3)\\ (k',6)=1}} \mu(k') 
\sum\limits_{\substack{(A',B')\in \mathbb{Z}^2\\ b^2A'^3+a^3B'^2 \equiv 0 \bmod{(k')^2}}} \Gamma\left(\frac{k_3A'}{X^4}\right)\Gamma\left(\frac{k_2B'}{X^6}\right),
$$
where $a = 3/k_3$ and $b = 2/k_2$. In particular it follows that $(ab,k')=1$ and $a,b\le 3$. We will need to
account for possible common factors of $A'B'$ and $k'$. Drawing out the greatest common divisor of
$B'$ and $k'$ we write $B' = hx$ and $k' = hl$, with $(x,l) = 1$. It easily follows from the square-freeness
of $k$ that $h|A'$ and we can write $A' = hy$ with $(hxy,l) = 1$. Hence,
\begin{equation*}
\begin{split}
S_1(X)= & \sum\limits_{k_2|2} \sum\limits_{k_3|3} \sum\limits_{\substack{h\le \xi/(k_2k_3) \\ (h,6)=1}}  \mu(k_2)\mu(k_3)\mu(h)\sum\limits_{\substack{ l \le \xi/(k_2k_3h) \\ (l,6h)=1}} \mu(l) \times\\ &  
\sum\limits_{\substack{(x,y)\in \mathbb{Z}^2\\ a^3x^2+b^2hy^3 \equiv 0 \bmod l^2\\(xy,l)=1}} \Gamma\left(\frac{k_2hx}{X^6}\right)\Gamma\left(\frac{k_3hy}{X^4}\right),
\end{split}
\end{equation*} 
with $(abh, l)=1$. Now applying the Poisson summation formula to both the sums over $x$ and $y$, we deduce that
\begin{equation} \label{afterpoisson}
\begin{split}
S_1(X)= & X^{10} \sum\limits_{k_2|2} \sum\limits_{k_3|3} \sum\limits_{\substack{h\le \xi/(k_2k_3) \\ (h,6)=1}} \frac {\mu(k_2)}{k_2}\cdot \frac{\mu(k_3)}{k_3}\cdot \frac{\mu(h)}{h^2} 
\cdot \sum\limits_{\substack{ l \le \xi/(k_2k_3h) \\ (l,6h)=1}} \frac{\mu(l)}{l^4}\times \\
&  \sum\limits_{m\in \mathbb{Z}} \sum\limits_{n\in \mathbb{Z}} 
\hat\Gamma\left(\frac{X^6m}{k_2hl^2}\right) \hat\Gamma\left(\frac{X^4n}{k_3hl^2}\right) \mathcal{E}(m,n;l^2),
\end{split}
\end{equation}
where
$$
\mathcal{E}(m,n;l^2)=\sum\limits_{\substack{c,d \bmod l^2\\ a^3c^2+b^2hd^3 \equiv 0 \bmod l^2\\ (cd,l)=1}} e\left(\frac{cm+dn}{l^2}\right).
$$
We note that the inner double sum appearing in \eqref{afterpoisson} is equal to 
$$
\mathscr{S}\left(0,0,\frac{X^6}{k_2h},\frac{X^4}{k_3h};a^3,b^2h;l^2\right)
$$
in the notation of \cite[Eq. (5.1)]{BB}. 
As seen in \cite[section 5]{BB}, we can write the above complete exponential sum in two variables, $\mathcal{E}(m,n;l^2)$,
as a complete exponential sum over a single variable, namely
\begin{equation} \label{exptrans}
\mathcal{E}(m,n;l^2)=E(a^3b^4h^2m,a^3b^2hn;l^2), 
\end{equation}
where we define
$$
E(c,d;q):=\sum\limits_{\substack{z=1\\ (z,q)=1}}^{q} e\left(\frac{cz^3-dz^2}{q}\right).
$$

\section{Explicit evaluation of exponential sums I}  
Recall that $(abh,l)=1=(l,6)$ and $l$ is square-free. The exponential sum in the last section can be evaluated explicitly as follows.
\begin{equation} \label{of}
\begin{split}
& E(a^3b^4h^2m,a^3b^2hn;l^2)\\ = & 
 \sum\limits_{\substack{x=1\\ (x,l)=1}}^l  \sum\limits_{y=1}^l e\left(\frac{a^3b^4h^2m(x+ly)^3-a^3b^2hn(x+ly)^2}{l^2}\right)\\
= & \sum\limits_{\substack{x=1\\ (x,l)=1}}^l e\left(\frac{a^3b^4h^2mx^3-a^3b^2hnx^2}{l^2}\right) \sum\limits_{y=1}^l e\left(\frac{3a^3b^4h^2mx^2-2a^3b^2hnx}{l}\cdot y\right)\\
= & l \sum\limits_{\substack{x=1\\ (x,l)=1\\ 3a^3b^4h^2mx^2-2a^3b^2hnx\equiv 0 \bmod{l}}}^l e\left(\frac{a^3b^4h^2mx^3-a^3b^2hnx^2}{l^2}\right)\\
= & l \sum\limits_{\substack{x=1\\ (x,l)=1\\ mx\equiv \overline{3} \cdot 2\overline{b}^2\overline{h}n \bmod{l}}}^l e\left(\frac{a^3b^4h^2mx^3-a^3b^2hnx^2}{l^2}\right).
\end{split}
\end{equation}
Let $d=(m,l)$, $m_1=m/d$ and $l_1=l/d$. We note that $(m_1,l_1)=1$ and $(l_1,d)=1$ by square-freeness of $l$. The linear congruence in the last line of \eqref{of} is solvable if and only if $d|n$. Hence, this exponential sum is non-zero if and only if 
$d|n$, in which case we set $n_1=n/d$. Now we use the multiplicativity of the exponential sums $E(c,d;q)$. By \cite[Lemma 6]{BB}, we have
$$
E(a^3b^4h^2m,a^3b^2hn;l^2)=E(a^3b^4h^2m\overline{l_1}^2,a^3b^2hn\overline{l_1}^2;d^2)E(a^3b^4h^2m\overline{d}^2,a^3b^2hn\overline{d}^2;l_1^2),
$$
where $\overline{l_1}$ is a multiplicative inverse of $l_1$ modulo $d^2$, and $\overline{d}$ is a multiplicative inverse of $d$ modulo $l_1^2$. We further deduce 
that
\begin{equation} \label{expsumev}
E(a^3b^4h^2m,a^3b^2hn;l^2)=dE(a^3b^4h^2m_1\overline{l_1}^2,a^3b^2hn_1\overline{l_1}^2;d)E(a^3b^4h^2m_1\overline{d},a^3b^2hn_1\overline{d};l_1^2).
\end{equation}

If $m_1=0$, then it follows that $m=0$, $d=l$ and $l_1=1$ and hence
\begin{equation} \label{expsum1}
E(a^3b^4h^2m,a^3b^2hn;l^2)=dE(0,a^3b^2hn_1;d)=d\varphi(e) E(0,a^3b^2hn_2;d_2),
\end{equation}
where $e=(n_1,d)$, $n_2=n_1/e$, $d_2=d/e$, and $\varphi(e)$ is the Euler totient function. The last exponential sum is a quadratic Gauss sum with coprimality constraint to square-free modulus $d_2$. Its precise value is
$$
E\left(0,a^3b^2hn_2;d_2\right)= f\left(a^3b^2hn_2;d_2\right),
$$
where 
$$
f(c;q):=\prod\limits_{p|q} \left(\epsilon_p\left(\frac{-cq/p}{p}\right)\sqrt{p}-1\right)
$$
with $\left(\frac{c}{p}\right)$ being the Legendre symbol and 
$$
\epsilon_p:=\begin{cases} 1 & \mbox{if } p\equiv 1 \bmod 4,\\  i & \mbox{if } p\equiv -1 \bmod 4. \end{cases}
$$
Altogether, if $m_1=0$, then 
\begin{equation} \label{expsum0}
E(a^3b^4h^2m,a^3b^2hn;l^2)= d\varphi(e)\cdot f\left(a^3b^2hn_2;d_2\right).
\end{equation}

Let us now assume that $m_1\not=0$. We leave the first exponential sum on the right-hand side of \eqref{expsumev} as it is and transform the second one as in \eqref{of}, leading to
\begin{equation*} 
\begin{split}
& E(a^3b^4h^2m_1\overline{d},a^3b^2hn_1\overline{d};l_1^2) \\
= & l_1 \sum\limits_{\substack{x=1\\ (x,l_1)=1\\ m_1\overline{d}x\equiv \overline{3} \cdot 2\overline{b}^2\overline{h}n_1\overline{d} \bmod{l_1}}}^{l_1} e\left(\frac{a^3b^4h^2m_1\overline{d}x^3-a^3b^2hn_1\overline{d}x^2}{l_1^2}\right) \\
= &  l_1 \sum\limits_{\substack{x=1\\ (x,l_1)=1\\ x\equiv \overline{3} \cdot 2\overline{b}^2\overline{h}n_1\overline{m_1} \bmod{l_1}}}^{l_1} e\left(\frac{a^3b^4h^2m_1\overline{d}x^3-a^3b^2hn_1\overline{d}x^2}{l_1^2}\right).
\end{split}
\end{equation*}
Using Hensel's Lemma, we can uniquely lift the solution $x$ to the congruence modulo $l_1$ in the last line to a solution of the same congruence modulo $l_1^2$. Then plugging this $x$ into the exponential term, and recalling that $a=3/k_3$ and $b=2/k_2$, we obtain
\begin{equation*}
E(a^3b^4h^2m_1\overline{d},a^3b^2hn_1\overline{d};l_1^2)=l_1 \cdot e\left(-\frac{\overline{hd} (k_2\overline{m_1})^2(\overline{k_3}n_1)^3}{l_1^2}\right).
\end{equation*}
Flipping the Kloosterman fractions now gives
\begin{equation*}
E(a^3b^4h^2m_1\overline{d},a^3b^2hn_1\overline{d};l_1^2)=l_1 \cdot e\left(-\frac{k_2^2n_1^3}{k_3^3hdm_1^2l_1^2}\right) \cdot e\left(\frac{\overline{l_1}^2k_2^2n_1^3}{k_3^3hdm_1^2}\right).
\end{equation*}
Combining this with \eqref{expsumev}, we obtain
\begin{equation} \label{expsum}
\begin{split}
E(a^3b^4h^2m,a^3b^2hn;l^2)= & l_1 \cdot E(a^3b^4h^2m_1\overline{l_1}^2,a^3b^2hn_1\overline{l_1}^2;d) \times\\ & 
e\left(-\frac{k_2^2n_1^3}{k_3^3hdm_1^2l_1^2}\right) \cdot e\left(\frac{\overline{l_1}^2k_2^2n_1^3}{k_3^3hdm_1^2}\right).
\end{split}
\end{equation}

Plugging \eqref{expsum0} and \eqref{expsum} into \eqref{afterpoisson}, we get
\begin{equation} \label{S2split}
S_1(X)=M(X)+E(X),
\end{equation}
where 
\begin{equation} \label{M}
\begin{split}
M(X)= & X^{10} \sum\limits_{k_2|2} \sum\limits_{k_3|3} \sum\limits_{\substack{h\le \xi/(k_2k_3) \\ (h,6)=1}} \sum\limits_{\substack{d_2\le \xi/(k_2k_3h)\\ (d_2,6h)=1}}
\sum\limits_{\substack{e\le \xi/(k_2k_3hd_2)\\ (e,6hd_2)=1}}\frac {\mu(k_2)}{k_2}\cdot \frac{\mu(k_3)}{k_3}\cdot \frac{\mu(h)}{h^2}\times\\ &
\frac{\mu(d_2)}{d_2^3} \cdot \frac{\mu(e)\varphi(e)}{e^3} \cdot \sum\limits_{n_2\in \mathbb{Z}}
\hat\Gamma\left(\frac{X^4n_2}{k_3hd_2}\right) \cdot f\left(a^3b^2hn_2;d_2\right)
\end{split}
\end{equation}
and 
\begin{equation} \label{E}
\begin{split}
E(X)= & X^{10} \sum\limits_{k_2|2} \sum\limits_{k_3|3} \sum\limits_{\substack{h\le \xi/(k_2k_3) \\ (h,6)=1}} \sum\limits_{\substack{d\le \xi/(k_2k_3h)\\ (d,6h)=1}} \frac {\mu(k_2)}{k_2}\cdot \frac{\mu(k_3)}{k_3}\cdot \frac{\mu(h)}{h^2}\cdot \frac{\mu(d)}{d^3} \times\\ & 
\sum\limits_{\substack{ l_1 \le \xi/(k_2k_3hd) \\ (l_1,6hd)=1}} \frac{\mu(l_1)}{l_1^3}\cdot
\sum\limits_{\substack{m_1\in \mathbb{Z}\setminus\{0\}\\ (m_1,l_1)=1}} 
\hat\Gamma\left(\frac{X^6m_1}{k_2hdl_1^2}\right) \sum\limits_{n_1\in \mathbb{Z}} 
\hat\Gamma\left(\frac{X^4n_1}{k_3hdl_1^2}\right) e\left(-\frac{k_2^2n_1^3}{k_3^3hdm_1^2l_1^2}\right) \times \\ & 
E(a^3b^4h^2m_1\overline{l_1}^2,a^3b^2hn_1\overline{l_1}^2;d) \cdot e\left(\frac{\overline{l_1}^2k_2^2n_1^3}{k_3^3hdm_1^2}\right).
\end{split}
\end{equation}
The term $M(X)$ will be the main term, and $E(X)$ an error term whose treatment will be the key part of this paper and carried out from section \ref{PS2} onwards. First we shall deal with $M(X)$.

\section{Evaluation of the main term}
We split $M(X)$ into two parts $M_0(X)$ and $E_0(X)$, $M_0(X)$ being the contribution of $n_2=0$, and $E_0(X)$ being the contribution of $n_2\not=0$. Hence, 
\begin{equation} \label{Msplit}
M(X)=M_0(X)+E_0(X),
\end{equation}
where 
\begin{equation} \label{M0}
\begin{split}
M_0(X)= & X^{10} \sum\limits_{k_2|2} \sum\limits_{k_3|3} \sum\limits_{\substack{h\le \xi/(k_2k_3) \\ (h,6)=1}} 
\sum\limits_{\substack{e\le \xi/(k_2k_3h)\\ (e,6h)=1}}\frac {\mu(k_2)}{k_2}\cdot \frac{\mu(k_3)}{k_3}\cdot \frac{\mu(h)}{h^2}\cdot \frac{\mu(e)\varphi(e)}{e^3} 
\end{split}
\end{equation}
and
\begin{equation*} 
\begin{split}
E_0(X)= & X^{10} \sum\limits_{k_2|2} \sum\limits_{k_3|3} \sum\limits_{\substack{h\le \xi/(k_2k_3) \\ (h,6)=1}} \sum\limits_{\substack{d_2\le \xi/(k_2k_3h)\\ (d_2,6h)=1}}
\sum\limits_{\substack{e\le \xi/(k_2k_3hd_2)\\ (e,6hd_2)=1}}\frac {\mu(k_2)}{k_2}\cdot \frac{\mu(k_3)}{k_3}\cdot \frac{\mu(h)}{h^2}\times\\ & \frac{\mu(d_2)}{d_2^3} \cdot 
\frac{\mu(e)\varphi(e)}{e^3} \cdot  \sum\limits_{\substack{n_2\in \mathbb{Z}\\ n_2\not=0}}
\hat\Gamma\left(\frac{X^4n_2}{k_3hd_2}\right) f\left(a^3b^2hn_2;d_2\right).
\end{split}
\end{equation*}

The term $M_0(X)$ will yield the main contribution and can be easily evaluated as follows. Completing the inner-most sum over $e$ on the right-hand of \eqref{M0} to a series, and writing this series as an Euler
product, we obtain
\begin{equation*}
\begin{split}
M_0(X)= & X^{10} \sum\limits_{k_2|2} \sum\limits_{k_3|3} \sum\limits_{\substack{h\le \xi/(k_2k_3) \\ (h,6)=1}}  
\frac{\mu(k_2)}{k_2}\cdot \frac{\mu(k_3)}{k_3} \cdot \frac{\mu(h)}{h^2} \cdot \left(\sum\limits_{\substack{ e=1 \\ (e,6h)=1}}^{\infty}  \frac{\mu(e)\varphi\left(e\right)}{e^3}
+O\left(\frac{h}{\xi}\right)\right)\\
= & X^{10} \sum\limits_{k_2|2} \sum\limits_{k_3|3}  
\frac{\mu(k_2)}{k_2}\cdot \frac{\mu(k_3)}{k_3} \cdot \sum\limits_{\substack{h\le \xi/(k_2k_3) \\ (h,6)=1}}  \frac{\mu(h)}{h^2} \cdot \prod\limits_{p|h} \left(1-\frac{p-1}{p^3}\right)^{-1} \times\\ &
 \prod\limits_{p> 3} \left(1-\frac{p-1}{p^3}\right) + O\left(\frac{X^{10}}{\xi}\right).
\end{split}
\end{equation*} 
Now completing the sum over $h$ to a series, and writing this series as an Euler product, we further deduce that
\begin{equation} \label{themain}
\begin{split}
M_0(X)= &
X^{10} \sum\limits_{k_2|2} \sum\limits_{k_3|3}\frac{\mu(k_2)}{k_2}\cdot \frac{\mu(k_3)}{k_3} \cdot \left(\sum\limits_{(h,6)=1}  \frac{\mu(h)}{h^2} \cdot \prod\limits_{p|h} \left(1-\frac{p-1}{p^3}\right)^{-1} +
O\left(\frac{1}{\xi}\right)\right)\times\\ &
 \prod\limits_{p> 3} \left(1-\frac{p-1}{p^3}\right) + O\left(\frac{X^{10}}{\xi}\right)\\
= & X^{10}  \sum\limits_{k_2|2} \sum\limits_{k_3|3} \frac{\mu(k_2)}{k_2}\cdot \frac{\mu(k_3)}{k_3} \cdot \prod\limits_{p>3}  \left(1-\frac{p}{p^3-p+1}\right) \cdot \prod\limits_{p> 3} \left(1-\frac{p-1}{p^3}\right)   +O\left(\frac{X^{10}}{\xi}\right)\\ 
= & X^{10} \cdot \frac{1}{3} \cdot  \prod\limits_{p>3}  \left(1-\frac{2p-1}{p^3}\right)    +O\left(\frac{X^{10}}{\xi}\right).
\end{split}
\end{equation}

\section{Estimation of $E_0(X)$}
We rearrange the summation to get
\begin{equation*}
\begin{split}
E_0(X)= & X^{10} \sum\limits_{k_2|2} \sum\limits_{k_3|3}  \sum\limits_{\substack{d_2\le \xi/(k_2k_3)\\ (d_2,6)=1}}
\sum\limits_{\substack{e\le \xi/(k_2k_3d_2)\\ (e,6d_2)=1}}\frac {\mu(k_2)}{k_2}\cdot \frac{\mu(k_3)}{k_3}\cdot \frac{\mu(d_2)}{d_2^3} \cdot \frac{\mu(e)\varphi(e)}{e^3} \times \\ &
\sum\limits_{\substack{n_2\in \mathbb{Z}\\ n_2\not=0}} \sum\limits_{\substack{h\le \xi/(k_2k_3d_2e)\\ (h,6k_2k_3d_2e)=1}}
\frac{\mu(h)}{h^2} \cdot \hat\Gamma\left(\frac{X^4n_2}{k_3hd_2}\right) f\left(a^3b^2hn_2;d_2\right).
\end{split}
\end{equation*}
Since $\hat\Gamma(x)$ has rapid decay, we can cut the summation over $n_2$ at $|n_2|<\xi X^{\varepsilon-4}$ and the summation over $h$ at $h>X^4n_2(k_3d_2)^{-1}X^{-\varepsilon}$ at the cost of an altogether error of size $O(1)$. Hence, taking $f\left(a^3b^2hn_2;d_2\right)\ll d_2^{1/2+\varepsilon}$ into account, we obtain
\begin{equation} \label{E0est}
\begin{split}
E_0(X)\ll & 1+X^{10+\varepsilon} \sum\limits_{k_2|2} \sum\limits_{k_3|3}  \sum\limits_{d_2\le \xi}
\sum\limits_{e\le \xi} \frac{1}{d_2^{5/2}} \cdot \frac{1}{e^2}
\sum\limits_{0<|n_2|\le \xi X^{\varepsilon-4}} 
\sum\limits_{h>X^4n_2(k_3d_2)^{-1}X^{-\varepsilon}} \frac{1}{h^2}\\
\ll &1+X^{10+\varepsilon}  \sum\limits_{d_2\le \xi}
\sum\limits_{e\le \xi} \frac{1}{d_2^{5/2}} \cdot \frac{1}{e^2}\cdot  \sum\limits_{0<|n_2|\le \xi X^{\varepsilon-4}} \frac{d_2}{X^4n_2}
=O\left(X^{6+\varepsilon}\right).
\end{split}
\end{equation}

\section{Partial estimation of $E(X)$}
It remains to estimate the error term $E(X)$, defined in \eqref{E}. In this section, we consider the contribution of large $hd$ to  $E(X)$. We shall need the following lemma which is a consequence of \cite[Lemma 10]{BB}. 

\begin{lemma} \label{1} Let $s,t\in \mathbb{Z}$ and $Q\in \mathbb{N}$. Then 
$$
E(s,t;Q)\ll (s,t,Q)^{1/2}Q^{1/2+\varepsilon}.
$$ 
\end{lemma}

Let $0<\kappa<1$ be a real parameter, to be fixed later. Let $E^+(X)$ be the contribution of $hd>\xi^{\kappa}$ and $E^-(X)$ that of $hd\le \xi^{\kappa}$ and hence
\begin{equation} \label{Esplit}
E(X)=E^+(X)+E^-(X).
\end{equation}
We now want to bound $E^+(X)$. By Lemma \ref{1}, we have 
$$
E(a^3b^4h^2m_1\overline{l_1}^2,a^3b^2hn_1\overline{l_1}^2;d)\ll (m_1,n_1,d)^{1/2}d^{1/2+\varepsilon}
$$
and hence
\begin{equation} \label{E+esti}
\begin{split}
E^+(X)= & X^{10+\varepsilon} \sum\limits_{h>\xi^{\kappa}} \sum\limits_{\substack{\xi^{\kappa}/h<d\le \xi/h \\ (d,h)=1}}  \frac{1}{h^2}\cdot \frac{1}{d^{3/2}} \cdot  
\sum\limits_{\substack{ l_1 \le \xi/(hd) \\ (l_1,d)=1}} \frac{1}{l_1^3}\times\\ & 
\sum\limits_{0<|m_1|\le X^{\varepsilon-6}hdl_1^2} \sum\limits_{|n_1|\le X^{\varepsilon-4}hdl_1^2} (m_1,n_1,d)^{1/2},
\end{split}
\end{equation}
where we use the fact that $\hat\Gamma$ has rapid decay. The inner-most double sum in \eqref{E+esti} is easily estimated by
$$
\sum\limits_{0<|m_1|\le X^{\varepsilon-6}hdl_1^2} \sum\limits_{|n_1|\le X^{\varepsilon-4}hdl_1^2} (m_1,n_1,d) \ll d^{\varepsilon} \cdot \frac{hdl_1^2}{X^{6-\varepsilon}} \cdot
\left(1+\frac{hdl_1^2}{X^{4-\varepsilon}}\right). 
$$
Now a short calculation yields
\begin{equation} \label{E+bound}
 E^+(X)\ll \xi^{(1-\kappa)/2}X^{4+\varepsilon}+\xi^{2-\kappa}X^{\varepsilon}.
\end{equation}

\section{Application of Poisson summation II} \label{PS2}
We are left with estimating
\begin{equation} \label{E-}
\begin{split}
E^-(X)= & X^{10} \sum\limits_{k_2|2} \sum\limits_{k_3|3} \mathop{\sum\limits_{\substack{h\le \xi/(k_2k_3) \\ (h,6)=1}}\sum\limits_{\substack{d\le \xi/(k_2k_3h)\\ (d,6h)=1}}}_{hd\le \xi^{\kappa}}  \frac {\mu(k_2)}{k_2}\cdot \frac{\mu(k_3)}{k_3}\cdot \frac{\mu(h)}{h^2}\cdot \frac{\mu(d)}{d^3} \times\\ & 
\sum\limits_{\substack{ l_1 \le \xi/(k_2k_3hd) \\ (l_1,6hd)=1}} \frac{\mu(l_1)}{l_1^3}\cdot
\sum\limits_{\substack{m_1\in \mathbb{Z}\setminus\{0\}\\ (m_1,l_1)=1}} 
\hat\Gamma\left(\frac{X^6m_1}{k_2hdl_1^2}\right) \sum\limits_{n_1\in \mathbb{Z}} 
\hat\Gamma\left(\frac{X^4n_1}{k_3hdl_1^2}\right) \times \\ &  e\left(-\frac{k_2^2n_1^3}{k_3^3hdm_1^2l_1^2}\right)
E(a^3b^4h^2m_1\overline{l_1}^2,a^3b^2hn_1\overline{l_1}^2;d) \cdot e\left(\frac{\overline{l_1}^2k_2^2n_1^3}{k_3^3hdm_1^2}\right).
\end{split}
\end{equation}
 The term
$$
 \hat\Gamma\left(\frac{X^4n_1}{k_3hdl_1^2}\right)\cdot  e\left(-\frac{k_2^2n_1^3}{k_3^3hdm_1^2l_1^2}\right)
$$
will be interpreted as a slowly oscillating weight function of $n_1$. We therefore define $\Psi : \mathbb{R}\rightarrow \mathbb{C}$ as 
$$
\Psi\left(\frac{X^4z}{k_3hdl_1^2}\right):=\hat\Gamma\left(\frac{X^4z}{k_3hdl_1^2}\right) \cdot e\left(-\frac{k_2^2z^3}{k_3^3hdm_1^2l_1^2}\right),
$$
i.e.
\begin{equation} \label{Psidef}
\Psi\left(z\right):=\hat\Gamma\left(z\right) e\left(-\frac{k_2^2h^2d^2l_1^4}{X^{12}m_1^2} \cdot z^3\right).
\end{equation}
Now breaking up the summation over $n_1$ into residue classes modulo $k_3^3hdm_1^2$ and using the Poisson summation formula again, we get
\begin{equation} \label{so}
\begin{split}
& \sum\limits_{n_1\in \mathbb{Z}} \Psi\left(\frac{X^4n_1}{k_3hdl_1^2}\right) 
E(a^3b^4h^2m_1\overline{l_1}^2,a^3b^2hn_1\overline{l_1}^2;d) \cdot e\left(\frac{\overline{l_1}^2k_2^2n_1^3}{k_3^3hdm_1^2}\right)\\
= & \frac{l_1^2}{X^4k_3^2m_1^2} \cdot \sum\limits_{u\in \mathbb{Z}} \hat\Psi\left(\frac{l_1^2u}{X^4k_3^2m_1^2}\right) \sum\limits_{v=1}^{k_3^3hdm_1^2}
E(a^3b^4h^2m_1\overline{l_1}^2,a^3b^2hv\overline{l_1}^2;d) \cdot  e\left(\frac{\overline{l_1}^2k_2^2v^3+uv}{k_3^3hdm_1^2}\right).
\end{split}
\end{equation}

We write $m_1=m^{\ast}\tilde{m}$, where rad$(m^{\ast})|(6hd)$ and $(\tilde{m},6hd)=1$, where rad$(n)$ is the largest square-free divisor of the natural number $n$. Further, we write $q:=k_3^3hd(m^{\ast})^2$ and note that $(\tilde{m},q)=1$. Hence we can write the inner-most sum on the right-hand side of \eqref{so} as 
\begin{equation} \label{reduction}
\begin{split}
& \sum\limits_{v=1}^{k_3^3hdm_1^2}
E(a^3b^4h^2m_1\overline{l_1}^2,a^3b^2hv\overline{l_1}^2;d) \cdot e\left(\frac{\overline{l_1}^2k_2^2v^3+uv}{k_3^3hdm_1^2}\right) \\
=  & \sum\limits_{x=1}^{q}\sum\limits_{y=1}^{\tilde{m}^2} E(a^3b^4h^2m^{\ast}\tilde{m}\overline{l_1}^2,a^3b^2h(x\tilde{m}^2+yq)\overline{l_1}^2;d) \cdot e\left(\frac{\overline{l_1}^2k_2^2(x\tilde{m}^2+yq)^3+u(x\tilde{m}^2+yq)}{q\tilde{m}^2}\right) \\
= & \sum\limits_{x=1}^{q} E(a^3b^4h^2m^{\ast}\tilde{m}\overline{l_1}^2,a^3b^2hx\tilde{m}^2\overline{l_1}^2;d) \cdot 
e\left(\frac{\overline{l_1}^2k_2^2\tilde{m}^4x^3+ux}{q}\right) \cdot
\sum\limits_{y=1}^{\tilde{m}^2} e\left(\frac{\overline{l_1}^2k_2^2q^2y^3+uy}{\tilde{m}^2}\right).
\end{split}
\end{equation}

Furthermore, we write $q=\tilde{q}d$ and 
\begin{equation} \label{reduction2}
\begin{split}
= & \sum\limits_{x=1}^{q} E(a^3b^4h^2m^{\ast}\tilde{m}\overline{l_1}^2,a^3b^2hx\tilde{m}^2\overline{l_1}^2;d) \cdot 
e\left(\frac{\overline{l_1}^2k_2^2\tilde{m}^4x^3+ux}{q}\right)\\
= & \sum\limits_{y=1}^{\tilde{q}} \sum\limits_{z=1}^{d} E(a^3b^4h^2m^{\ast}\tilde{m}\overline{l_1}^2,a^3b^2h(dy+z)\tilde{m}^2\overline{l_1}^2;d) \cdot 
e\left(\frac{\overline{l_1}^2k_2^2\tilde{m}^4(dy+z)^3+u(dy+z)}{q}\right) \\
= & \sum\limits_{z=1}^{d} E(a^3b^4h^2m^{\ast}\tilde{m}\overline{l_1}^2,a^3b^2hz\tilde{m}^2\overline{l_1}^2;d) \cdot e\left(\frac{\overline{l_1}^2k_2^2\tilde{m}^4z^3+uz}{\tilde{q}}\right) \times\\ & \sum\limits_{y=1}^{\tilde{q}}
e\left(\frac{\overline{l_1}^2k_2^2\tilde{m}^4((dy+z)^3-z^3)+udy}{q}\right) \\
= & \sum\limits_{z=1}^{d} E(a^3b^4h^2m^{\ast}\tilde{m}\overline{l_1}^2,a^3b^2hz\tilde{m}^2\overline{l_1}^2;d) \cdot e\left(\frac{\overline{l_1}^2k_2^2\tilde{m}^4z^3+uz}{\tilde{q}}\right) \times\\ &  \sum\limits_{y=1}^{\tilde{q}}
e\left(\frac{\overline{l_1}^2k_2^2\tilde{m}^4(d^2y^3+3dy^2z+3yz^2)+uy}{\tilde{q}}\right)
\end{split}
\end{equation}

Combining \eqref{E-}, \eqref{so}, \eqref{reduction} and \eqref{reduction2}, we obtain
\begin{equation} \label{Enew}
\begin{split}
& E^-(X)\\ = & X^{6} \sum\limits_{k_2|2} \sum\limits_{k_3|3} \mathop{\sum\limits_{\substack{h\le \xi/(k_2k_3) \\ (h,6)=1}} \sum\limits_{\substack{d\le \xi/(k_2k_3h)\\ (d,6h)=1}}}_{hd\le \xi^{\kappa}} \frac {\mu(k_2)}{k_2}\cdot \frac{\mu(k_3)}{k_3^3}\cdot \frac{\mu(h)}{h^2}\cdot \frac{\mu(d)}{d^3} \cdot
\sum\limits_{\substack{ l_1 \le \xi/(k_2k_3hd) \\ (l_1,6hd)=1}} \frac{\mu(l_1)}{l_1} \times\\ & 
\sum\limits_{\substack{m^{\ast}\in \mathbb{Z}\setminus\{0\}\\ (m^{\ast},l_1)=1\\ \mbox{\tiny rad}(m^{\ast})|6hd}} \frac{1}{(m^{\ast})^2} \cdot
\sum\limits_{\substack{\tilde{m} \in \mathbb{Z}\setminus\{0\}\\ (\tilde{m},6hdl_1)=1}} \frac{1}{\tilde{m}^2} \cdot 
\hat\Gamma\left(\frac{X^6m^{\ast}\tilde{m}}{k_2hdl_1^2}\right) \cdot
\sum\limits_{u\in \mathbb{Z}} \hat\Psi\left(\frac{l_1^2u}{X^4k_3^2|m^{\ast}\tilde{m}|^2}\right)   \times\\ &  
F(\overline{l_1}^2k_2^2q^2,0,u;\tilde{m}^2)\cdot  \sum\limits_{z=1}^{d} E(a^3b^4h^2m^{\ast}\tilde{m}\overline{l_1}^2,a^3b^2hz\tilde{m}^2\overline{l_1}^2;d) \cdot e\left(\frac{\overline{l_1}^2k_2^2\tilde{m}^4z^3+uz}{\tilde{q}}\right) \times\\ &
F\left(\overline{l_1}^2k_2^2\tilde{m}^4d^2,3\overline{l_1}^2k_2^2\tilde{m}^4dz,3\overline{l_1}^2k_2^2\tilde{m}^4z^2+u;\tilde{q}\right),
\end{split}
\end{equation}
where 
\begin{equation} \label{F}
F(c_3,c_2,c_1;r):=\sum\limits_{x=1}^{r} e\left(\frac{c_3x^3+c_2x^2+c_1x}{r}\right).
\end{equation}

\section{Estimation of exponential sums}
We shall need bounds for the two exponential sums appearing in the inner-most sum in \eqref{Enew}. By Lemma \ref{1}, we have
\begin{equation} \label{dies} 
E(a^3b^4h^2m^{\ast}\tilde{m}\overline{l_1}^2,a^3b^2hz\tilde{m}^2\overline{l_1}^2;d) \ll d^{1/2+\varepsilon} \left(a^3b^2hz\tilde{m}^2\overline{l_1}^2,d\right)^{1/2} =
d^{1/2+\varepsilon}(z,d)^{1/2}.
\end{equation}
To bound the second exponential sum, we recall \cite[Lemma  4.1]{BB} which is due to Loxton and Schmidt \cite{LS}. 

\begin{lemma} \label{2} Let $Q\in \mathbb{N}$ and $f\in \mathbb{Z}[X]$. Suppose that $f'$ has degree $n$, precisely $m$ distinct roots and factorization
$$
f'(X)=A(X-\zeta_1)^{\eta_1}(X-\zeta_2)^{\eta_2}\cdots (X-\zeta_m)^{\eta_m}.
$$  
Define the semi-discriminant of $f'$ to be 
$$
\Delta=\Delta(f'):=A^{2n-2}\prod\limits_{i\not=j} (\zeta_i-\zeta_j)^{\eta_i\eta_j}
$$
and the exponent of $f'$ to be
$$
\eta=\eta(f'):=\max\{\eta_1,...,\eta_m\}.
$$
Then 
$$
\sum\limits_{x=1}^Q e\left(\frac{f(x)}{Q}\right) \le Q^{1-1/(2\eta)}(\Delta,Q)^{1/(2\eta)}n^{\omega(Q)},
$$
where $\omega(Q)$ is the number of distinct prime factors of $Q$. 
\end{lemma}

We further need the following simple observation.

\begin{lemma} \label{3} Let $Q\in \mathbb{N}$ and $f(X)=c_nX^n+c_{n-1}X^{n-1}+...+c_1X\in \mathbb{Z}[X]$. Suppose that $\delta | (c_n,...,c_2,Q)$. Then
$$
\sum\limits_{x=1}^Q e\left(\frac{f(x)}{Q}\right)\not=0 \Longrightarrow \delta|c_1 \mbox{ and } \sum\limits_{x=1}^Q e\left(\frac{f(x)}{Q}\right)=
\delta \sum\limits_{x=1}^{\tilde{Q}} e\left(\frac{\tilde{f}(X)}{\tilde{Q}}\right),
$$ 
where $\tilde{f}(X)=f(X)/\delta\in \mathbb{Z}[X]$ and $\tilde{Q}=Q/\delta\in \mathbb{N}$. 
\end{lemma}

\begin{proof} By the conditions in Lemma \ref{3}, we have
\begin{equation*}
\begin{split}
\sum\limits_{x=1}^Q e\left(\frac{f(x)}{Q}\right)= & \sum\limits_{y=1}^{\tilde{Q}} \sum\limits_{z=1}^{\delta} e\left(\frac{f(z\tilde{Q}+y)-c_1(z\tilde{Q}+y)}{Q}\right) e\left(\frac{c_1(z\tilde{Q}+y)}{Q}\right) \\
= &  \sum\limits_{y=1}^{\tilde{Q}} e\left(\frac{f(y)}{Q}\right) \sum\limits_{z=1}^{\delta} e\left(\frac{c_1z}{\delta}\right) \\
= & \sum\limits_{y=1}^{\tilde{Q}} e\left(\frac{\tilde{f}(y)}{\tilde{Q}}\right)\cdot \begin{cases} \delta & \mbox{ if } \delta|c_1\\ 0 & \mbox{ if } \delta\nmid c_1. \end{cases}
\end{split}
\end{equation*}
\end{proof}

Now we are ready to estimate $F\left(\overline{l_1}^2k_2^2\tilde{m}^4d^2,3\overline{l_1}^2k_2^2\tilde{m}^4dz,3\overline{l_1}^2k_2^2\tilde{m}^4z^2+u;\tilde{q}\right)$.  
We set $\delta:=(d,|m^{\ast}|)$. Then from Lemma \ref{3}, we deduce that
\begin{equation} \label{oje}
\begin{split}
& F\left(\overline{l_1}^2k_2^2\tilde{m}^4d^2,3\overline{l_1}^2k_2^2\tilde{m}^4dz,3\overline{l_1}^2k_2^2\tilde{m}^4z^2+u;\tilde{q}\right)\not=0 \Longrightarrow \delta|\left(3\overline{l_1}^2k_2^2\tilde{m}^4z^2+u\right)
\mbox{ and } \\
& F\left(\overline{l_1}^2k_2^2\tilde{m}^4d^2,3\overline{l_1}^2k_2^2\tilde{m}^4dz,3\overline{l_1}^2k_2^2\tilde{m}^4z^2+u;\tilde{q}\right)\\ = & \delta 
F\left(\overline{l_1}^2k_2^2\tilde{m}^4d^2\delta^{-1},3\overline{l_1}^2k_2^2\tilde{m}^4dz\delta^{-1},(3\overline{l_1}^2k_2^2\tilde{m}^4z^2+u)\delta^{-1};\tilde{q}\delta^{-1}\right).
\end{split}
\end{equation}
We assume that this is the case. Then we compute that the derivative of the polynomial
$$
f(X):=\overline{l_1}^2k_2^2\tilde{m}^4d^2\delta^{-1}X^3+3\overline{l_1}^2k_2^2\tilde{m}^4dz\delta^{-1}X^2+\left(3\overline{l_1}^2k_2^2\tilde{m}^4z^2+u\right)\delta^{-1}X \in \mathbb{Z}[X]
$$
has exponent 
$$
\eta(f')=\begin{cases} 2 & \mbox{ if } u=0\\ 1 & \mbox{ if } u\not=0 \end{cases}
$$ 
and semi-discriminant
$$
\Delta(f')= \begin{cases} \left(\overline{l_1}^2k_2^2\tilde{m}^4d^2\delta^{-1}\right)^2 & \mbox{ if } u=0 \\
12\overline{l_1}^2k_2^2\tilde{m}^4d^2\delta^{-2}u & \mbox{ if } u\not=0. \end{cases}
$$
Now from \eqref{oje} and Lemma \ref{2}, it follows that
\begin{equation} \label{dasdas}
\begin{split}
F\left(\overline{l_1}^2k_2^2\tilde{m}^4d^2,3\overline{l_1}^2k_2^2\tilde{m}^4dz,3\overline{l_1}^2k_2^2\tilde{m}^4z^2+u;\tilde{q}\right) \ll & \delta^{1/4} \tilde{q}^{3/4+\varepsilon}
\left(\left(\overline{l_1}^2k_2^2\tilde{m}^4d^2\delta^{-1}\right)^2,\tilde{q}\delta^{-1}\right )^{1/4}\\
\ll & \delta^{3/4} \tilde{q}^{3/4+\varepsilon} = (d,|m^{\ast}|)^{3/4} \tilde{q}^{3/4} \mbox{ if } u=0
\end{split}
\end{equation}
and 
\begin{equation} \label{das} 
\begin{split}
& F\left(\overline{l_1}^2k_2^2\tilde{m}^4d^2,3\overline{l_1}^2k_2^2\tilde{m}^4dz,3\overline{l_1}^2k_2^2\tilde{m}^4z^2+u;\tilde{q}\right)\\ \ll & \delta^{1/2} \tilde{q}^{1/2+\varepsilon}
\left(12\overline{l_1}^2k_2^2\tilde{m}^4(d\delta^{-1})^2u,\tilde{q}\delta^{-1}\right )^{1/2}\\  \le &
 \delta^{1/2} \tilde{q}^{1/2+\varepsilon} (u,\tilde{q})^{1/2} \ll (d,|m^{\ast}|)^{1/2} \tilde{q}^{1/2+\varepsilon} \left(u,hd|m^{\ast}|^2\right)^{1/2} \mbox{ if } u\not=0,
 \end{split}
\end{equation}
where we note that $(d\delta^{-1})^2$ and $\tilde{q}\delta^{-1}$ are coprime. 
Putting \eqref{dies} and \eqref{das} together and summing over $z$, we see that the inner-most sum in \eqref{Enew} is bounded by
\begin{equation} \label{EF}
\begin{split}
& \sum\limits_{z=1}^{d} E(a^3b^4h^2m^{\ast}\tilde{m}\overline{l_1}^2,a^3b^2hz\tilde{m}^2\overline{l_1}^2;d) \cdot e\left(\frac{\overline{l_1}^2k_2^2\tilde{m}^4z^3+uz}{\tilde{q}}\right)\times\\ &  F\left(\overline{l_1}^2k_2^2\tilde{m}^4d^2,3\overline{l_1}^2k_2^2\tilde{m}^4d,3\overline{l_1}^2k_2^2\tilde{m}^4z^2+u;\tilde{q}\right) \\
 \ll & d^{1/2+\varepsilon}\tilde{q}^{1/2+\varepsilon}(d,|m^{\ast}|)^{1/2}\left(u, hd(m^{\ast})^2\right)^{1/2}\sum\limits_{z=1}^d (z,d)^{1/2}\\ 
 \le & \left(h^{1/2}d^{3/2}|m^{\ast}|\right)^{1+\varepsilon} (d,|m^{\ast}|)^{1/2} \left(u, hd|m^{\ast}|^2\right)^{1/2} \mbox{ if } u\not=0.
 \end{split}
\end{equation}
Similarly, we deduce
\begin{equation} \label{EF0}
\begin{split}
& \sum\limits_{z=1}^{d} E(a^2b^4h^2m^{\ast}\tilde{m}\overline{l_1}^2,a^3b^2hz\tilde{m}^2\overline{l_1}^2;d) \cdot e\left(\frac{\overline{l_1}^2k_2^2\tilde{m}^4z^3+uz}{\tilde{q}}\right)\times\\ &  F\left(\overline{l_1}^2k_2^2\tilde{m}^4d^2,3\overline{l_1}^2k_2^2\tilde{m}^4d,3\overline{l_1}^2k_2^2\tilde{m}^4z^2;\tilde{q}\right) \\
 \ll & d^{1/2+\varepsilon}\tilde{q}^{3/4+\varepsilon}(d,|m^{\ast}|)^{3/4}\sum\limits_{z=1}^d (z,d)^{1/2}\\ 
 \le & \left(h^{3/4}d^{7/4}|m^{\ast}|^{3/2}\right)^{1+\varepsilon} (d,|m^{\ast}|)^{3/4} \mbox{ if } u=0.
 \end{split}
\end{equation}

\section{Explicit evaluation of exponential sums II}
Now we turn to the key point of this paper, an explicit evaluation of the cubic exponential sum $F(\overline{l_1}^2k_2^2q^2,0,u;\tilde{m}^2)$ appearing in \eqref{Enew}, followed by
an averaging over $l_1$. We write 
\begin{equation} \label{Fev}
\begin{split}
F(\overline{l_1}^2k_2^2q^2,0,u;\tilde{m}^2)= & \sum\limits_{x=1}^{|\tilde{m}|}  \sum\limits_{y=1}^{|\tilde{m}|} 
e\left(\frac{\overline{l_1}^2k_2^2q^2(y\tilde{m}+x)^3+u(y\tilde{m}+x)}{\tilde{m}^2}\right)\\
= &  \sum\limits_{x=1}^{|\tilde{m}|}  
e\left(\frac{\overline{l_1}^2k_2^2q^2x^3+ux}{\tilde{m}^2}\right)  \sum\limits_{y=1}^{|\tilde{m}|} 
e\left(\frac{3\overline{l_1}^2k_2^2q^2x^2+u}{\tilde{m}}\cdot y\right) \\
= & |\tilde{m}| \sum\limits_{\substack{x=1\\ 3\overline{l_1}^2k_2^2q^2x^2+u \equiv 0 \bmod{|\tilde{m}|}}}^{|\tilde{m}|}
e\left(\frac{\overline{l_1}^2k_2^2q^2x^3+ux}{\tilde{m}^2}\right) \\
= & |\tilde{m}| \sum\limits_{\substack{x=1\\ x^2\equiv -\overline{3}u (\overline{k_2q}l_1)^2 \bmod{|\tilde{m}|}}}^{|\tilde{m}|}
e\left(\frac{((k_2q\overline{l_1})^2x^2+u)x}{\tilde{m}^2}\right)\\
= & |\tilde{m}| \sum\limits_{\substack{x_1=1\\ x_1^2\equiv -\overline{3}u \bmod{|\tilde{m}|}}}^{|\tilde{m}|}
e\left(\frac{((k_2q\overline{l_1})^2(\overline{k_2q}l_1x_1)^2+u)\overline{k_2q}l_1x_1}{\tilde{m}^2}\right)\\
= & |\tilde{m}| \sum\limits_{\substack{x_1=1\\ x_1^2\equiv -\overline{3}u \bmod{|\tilde{m}|}}}^{|\tilde{m}|}
e\left(\frac{(x_1^3+ux_1)\overline{k_2q}}{\tilde{m}^2}\cdot l_1\right),
\end{split}
\end{equation}
where the multiplicative inverses are modulo $\tilde{m}^2$. 

\section{Asymptotic estimation of exponential integrals}
We also need to estimate asymptotically the Fourier transform $\hat\Psi(z)$ of the function $\Psi(z)$ defined in \eqref{Psidef}.  We have
\begin{equation} \label{Idef}
\begin{split}
\hat\Psi(\alpha)=I(\alpha,\beta):=\int\limits_{-\infty}^{\infty} \Psi\left(z\right) e(-\alpha z)dz =  \int\limits_{-\infty}^{\infty} 
\hat\Gamma\left(z\right) e\left(-\beta z^3-\alpha z \right)dz,
\end{split}
\end{equation}
where $\beta$ is of the form
\begin{equation} \label{betadef}
\beta:=\frac{k_2^2h^2d^2s^4}{X^{12}|m^{\ast}\tilde{m}|^2}
\end{equation}
with $s=l_1$.
We shall be interested in values of $\alpha$ of the form
\begin{equation} \label{alphadef}
\alpha:=\frac{us^2}{X^4k_3^2|m^{\ast}\tilde{m}|^2}.
\end{equation}
In \cite[subsection 3.2.]{BB}, we provided estimates for $I(\alpha,\beta)$ for the special case when $\hat\Gamma(z)$ is a Gaussian, i.e.
$$
\hat\Gamma(z)=e^{-\pi z^2}.
$$
These results can be carried over to any Schwartz class function with minor modifications. We prove the following general estimates, similar to those in \cite[subsection 3.2.]{BB}, using standard estimates for exponential integrals.

\begin{lemma} \label{expintev} 
Suppose that 
$\alpha$ and $\beta>0$ are real numbers. Set 
$$
\delta(\alpha)=\begin{cases} 1 & \mbox{ if } \alpha<0,\\ 0 & \mbox{ if } \alpha\ge 0 \end{cases}
$$
and 
\begin{equation}
\begin{split}
G(\alpha,\beta):= & \delta(\alpha) \cdot \left(
\hat\Gamma\left(\frac{|\alpha|^{1/2}}{(3\beta)^{1/2}}\right) \cdot \frac{1}{2^{1/2}(3|\alpha|\beta)^{1/4}} \cdot e\left(\frac{1}{8}-\frac{2|\alpha|^{3/2}}{3^{3/2}\beta^{1/2}}\right) + 
\right.\\ & \left. \hat\Gamma\left(-\frac{|\alpha|^{1/2}}{(3\beta)^{1/2}}\right) \cdot \frac{1}{2^{1/2}(3|\alpha|\beta)^{1/4}} \cdot e\left(-\frac{1}{8}+\frac{2|\alpha|^{3/2}}{3^{3/2}\beta^{1/2}}\right)\right) 
\end{split}
\end{equation}
and define the function $I(\alpha,\beta)$ as in \eqref{Idef}. Let $\Delta>1$, $C>0$ and suppose that $\beta\ge \Delta^{-1}$. Then we have the estimates
\begin{eqnarray} \label{easy}
I(\alpha,\beta) & = & O(1)  \mbox{ if } \alpha=0,\\
\label{station}
I(\alpha,\beta) & = & G(\alpha,\beta)+ O\left(\Delta\log(2+\beta)\cdot |\alpha|^{-1}\right) \mbox{ if } 0<|\alpha|\le \Delta^2 \beta, \\
\label{intparts}
I(\alpha,\beta) & = & O\left(\Delta^{-C}\right) \mbox{ if }  \Delta^2\beta < \alpha.
\end{eqnarray}
\end{lemma} 

\begin{proof} The estimate \eqref{easy} follows by bounding the integral in question trivially by
\begin{equation} \label{trivial}
I(\alpha,\beta)\ll \int\limits_{-\infty}^{\infty} \left|\hat\Gamma(z)\right|\ dz = O\left(1\right),
\end{equation} 
and \eqref{intparts} follows by iterated integration by parts, saving a factor of size $\gg \Delta$ in each step. If $\alpha>0$, then $G(\alpha,\beta)=0$ in \eqref{station}, and the estimate in \eqref{station} 
follows from \cite[Lemma 8.10]{IK} for $k=1$ using integration by parts upon noting that 
$$
\left|\frac{d}{dz}\left(-\beta z^3 -\alpha z\right) \right| = 3\beta z^2+\alpha\ge \alpha.
$$
If $\alpha<0$, then we are in the stationary phase case with stationary points
$$
x_0:=\pm\frac{|\alpha|^{1/2}}{(3\beta)^{1/2}}.
$$ 
Set 
$$
a:=\frac{|\alpha|^{1/2}}{2\beta^{1/2}} \quad \mbox{ and } \quad b:= \frac{|\alpha|^{1/2}}{\beta^{1/2}}.
$$
We first deal with the partial integral over the intervall $[a,b]$.
Employing \cite[Corollary 8.15]{IK} together with \cite[Lemma 8.10]{IK} for $k=2$, and using integration by parts, we asymptotically evaluate this integral as 
\begin{equation} \label{>}
\begin{split}
& \int\limits_{a}^{b} \hat\Gamma(z) e\left(-\beta z^3-\alpha z\right) dz -  \hat\Gamma\left(\frac{|\alpha|^{1/2}}{(3\beta)^{1/2}}\right) \cdot \frac{1}{2^{1/2}(3|\alpha|\beta)^{1/4}} \cdot e\left(\frac{1}{8}-\frac{2|\alpha|^{3/2}}{3^{3/2}\beta^{1/2}}\right)\\
= & O\left(\log(2+\beta)\cdot (|\alpha|\beta)^{-1/2}\right) = O\left(\Delta\log(2+\beta)\cdot |\alpha|^{-1}\right),
\end{split}
\end{equation}
where we recall that $|\alpha|\le \Delta^2 \beta$. Similarly, we find
\begin{equation} \label{<}
\begin{split}
& \int\limits_{-b}^{-a} \hat\Gamma(z) e\left(-\beta z^3-\alpha z\right) \ dz - \hat\Gamma\left(-\frac{|\alpha|^{1/2}}{(3\beta)^{1/2}}\right) \cdot \frac{1}{2^{1/2}(3|\alpha|\beta)^{1/4}} \cdot e\left(-\frac{1}{8}+\frac{2|\alpha|^{3/2}}{3^{3/2}\beta^{1/2}}\right) \\
= & O\left(\Delta\log(2+\beta)\cdot  |\alpha|^{-1}\right).
\end{split}
\end{equation}
Using \eqref{>} and \eqref{<}, and estimating the remaining integrals over $(-\infty,-b)$, $(-a,a)$ and $(b,\infty)$ again using \cite[Lemma 8.10]{IK} for $k=1$, we obtain \eqref{station}.
\end{proof}

Since we shall apply partial summation over $l_1$, we shall also need the following asymptotic evaluation for 
\begin{equation} \label{I1alphabeta}
\frac{\partial}{\partial s} I(\alpha,\beta) = I_1(\alpha,\beta):=2\pi i \int\limits_{-\infty}^{\infty} 
\hat\Gamma\left(z\right)\cdot \left(-\frac{4\beta}{s} \cdot z^3-\frac{2\alpha}{s}\cdot z\right)\cdot e\left(-\beta z^3-\alpha z \right)dz,
\end{equation}
with $\alpha$ and $\beta$ being defined as in \eqref{betadef} and \eqref{alphadef}. The following result can be proved in a similar way as Lemma \ref{expintev}, where 
$$
\hat\Gamma\left(z\right)\cdot \left(-\frac{4\beta}{s} \cdot z^3-\frac{2\alpha}{s}\cdot z\right)
$$
now takes the rule of $\hat\Gamma(z)$.

\begin{lemma} \label{expintev1} 
Suppose that 
$\alpha$ and $\beta>0$ are real numbers. Set 
$$
\delta(\alpha)=\begin{cases} 1 & \mbox{ if } \alpha<0,\\ 0 & \mbox{ if } \alpha\ge 0 \end{cases}
$$
and 
\begin{equation}
\begin{split}
G_1(\alpha,\beta):= & \delta(\alpha) \cdot \left(
-\hat\Gamma\left(\frac{|\alpha|^{1/2}}{(3\beta)^{1/2}}\right) \cdot \frac{2^{1/2}|\alpha|^{5/4}}{3^{7/4}\beta^{3/4}s}
 \cdot e\left(\frac{1}{8}-\frac{2|\alpha|^{3/2}}{3^{3/2}\beta^{1/2}}\right) +
\right.\\ & \left. \hat\Gamma\left(-\frac{|\alpha|^{1/2}}{(3\beta)^{1/2}}\right) \cdot \frac{2^{1/2}|\alpha|^{5/4}}{3^{7/4}\beta^{3/4}s} \cdot e\left(-\frac{1}{8}+\frac{2|\alpha|^{3/2}}{3^{3/2}\beta^{1/2}}\right)\right) 
\end{split}
\end{equation}
and define the function $I_1(\alpha,\beta)$ as in \eqref{Idef}. Let $\Delta>1$, $C>0$ and suppose that $\beta\ge \Delta^{-1}$. Then we have the estimates
\begin{eqnarray*} 
I_1(\alpha,\beta) & = & G_1(\alpha,\beta)+ O\left(\Delta\log(2+\beta)\cdot s^{-1}\right) \mbox{ if } 0<|\alpha|\le \Delta^2 \beta, \\
I_1(\alpha,\beta) & = & O\left(\Delta^{-C}\right) \mbox{ if }  \Delta^2\beta < \alpha.
\end{eqnarray*}
\end{lemma} 

We note that for $\alpha$ and $\beta$ as in \eqref{betadef} and \eqref{alphadef}, we have
\begin{equation} \label{G}
\begin{split}
& G(\alpha,\beta)= \delta(u) \times\\ & \left(\hat\Gamma\left(\frac{X^4|u|^{1/2}}{3^{1/2}k_2k_3hds}\right) \cdot \frac{X^4k_3^{1/2}|m^{\ast}\tilde{m}|}{12^{1/4}(k_2hd)^{1/2}|u|^{1/4}s^{3/2}} \cdot e\left(\frac{1}{8}-\frac{2|u|^{3/2}s}{3^{3/2}k_3^3k_2hd|m^{\ast}\tilde{m}|^2}\right) + 
\right.\\ & \left.\hat\Gamma\left(-\frac{X^4|u|^{1/2}}{3^{1/2}k_2k_3hds}\right) \cdot \frac{2^{1/2}X^4k_3^{1/2}|m^{\ast}\tilde{m}|}{12^{1/4}(k_2hd)^{1/2}|u|^{1/4}s^{3/2}} \cdot e\left(-\frac{1}{8}+\frac{2|u|^{3/2}s}{3^{3/2}k_3^3k_2hd|m^{\ast}\tilde{m}|^2}\right)\right)
\end{split}
\end{equation}
and 
\begin{equation} \label{G1}
\begin{split}
& G_1(\alpha,\beta)= 2\pi i \cdot \delta(u) \times\\ & \left(-\hat\Gamma\left(\frac{X^4|u|^{1/2}}{3^{1/2}k_2k_3hds}\right) \cdot \frac{2^{1/2}X^4|u|^{5/4}}{3^{7/4}(k_2hd)^{3/2}k_3^{5/2}
|m^{\ast}\tilde{m}|s^{3/2}}
\cdot e\left(\frac{1}{8}-\frac{2|u|^{3/2}s}{3^{3/2}k_3^3k_2hd|m^{\ast}\tilde{m}|^2}\right) +
\right.\\ & \left.\hat\Gamma\left(-\frac{X^4|u|^{1/2}}{3^{1/2}k_2k_3hds}\right) \cdot \frac{2^{1/2}X^4|u|^{5/4}}{3^{7/4}(k_2hd)^{3/2}k_3^{5/2}
|m^{\ast}\tilde{m}|s^{3/2}} \cdot e\left(-\frac{1}{8}+\frac{2|u|^{3/2}s}{3^{3/2}k_3^3k_2hd|m^{\ast}\tilde{m}|^2}\right)\right).
\end{split}
\end{equation}
Further, we observe that
\begin{equation}
\frac{\partial}{\partial s} G(\alpha,\beta)= G_1(\alpha,\beta)+O\left(\frac{X^4|m^{\ast}\tilde{m}|}{(hd)^{1/2}|u|^{1/4}s^{5/2}}+\frac{X^8|u|^{1/4}|m^{\ast}\tilde{m}|}{(hd)^{3/2}s^{7/2}}\right).
\end{equation}

Now we suppose that 
\begin{equation} \label{holds}
s\ge \left(\frac{|m^{\ast}\tilde{m}}{k_2hd}\right)^{1/2}X^{3-\varepsilon}
\end{equation}
so that $\beta$, as specified in \eqref{betadef}, satisfies $\beta\ge X^{-\varepsilon}$. Further, we set 
$$
K:=\frac{(k_2k_3hds)^2}{X^{8-2\varepsilon}}
$$
and define
\begin{equation} \label{omegadef}
\Omega(\alpha,\beta)=I(\alpha,\beta)-G(\alpha,\beta).
\end{equation}
Then, for $\alpha$ and $\beta$ as in \eqref{betadef} and \eqref{alphadef}, it follows that
\begin{equation} \label{Omega}
\Omega(\alpha,\beta)
= \begin{cases} O(1) & \mbox{ if } \alpha=u=0, \\ 
O\left(X^{4+2\varepsilon}|m^{\ast}\tilde{m}|^2|u|^{-1}s^{-2}\right) & \mbox{ if } 0<|u|\le K, \\
O\left(X^{-C}\right) &  \mbox{ if } |u| > K,
\end{cases}
\end{equation}
\begin{equation} \label{Omegaprime}
\frac{\partial}{\partial s} \Omega(\alpha,\beta)
= \begin{cases} 
O(X^{2\varepsilon}(s^{-1}+X^4|m^{\ast}\tilde{m}|(hd)^{-1/2}|u|^{-1/4}s^{-5/2}\\
 +X^8|u|^{1/4}|m^{\ast}\tilde{m}|(hd)^{-3/2}s^{-7/2})) & \mbox{ if } 0<|u|\le K, \\
O_{\hat\Gamma,C}\left(X^{-C}\right) & \mbox{ if } |u| > K.
\end{cases}
\end{equation}
We also note that
\begin{equation} \label{Gest}
G(\alpha,\beta)=O\left(X^{-C}\right) \mbox{ if } |u| > K
\end{equation}
by rapid decay of $\hat\Gamma$. The last three estimates and equations are the results on exponential integrals we shall work with in the following.

\section{Rearranging summations and partial estimation of $E^{-}(X)$}
Now we pull in the sum over $l_1$ in \eqref{Enew}, getting
\begin{equation} \label{Enew1}
\begin{split}
E^-(X) = & X^{6} \sum\limits_{k_2|2} \sum\limits_{k_3|3} \mathop{\sum\limits_{\substack{h\le \xi/(k_2k_3) \\ (h,6)=1}} \sum\limits_{\substack{d\le \xi/(k_2k_3h)\\ (d,6h)=1}}}_{hd\le \xi^{\kappa}} \frac {\mu(k_2)}{k_2}\cdot \frac{\mu(k_3)}{k_3^3}\cdot \frac{\mu(h)}{h^2}\cdot \frac{\mu(d)}{d^3} \times\\ & 
\sum\limits_{\substack{m^{\ast}\in \mathbb{Z}\setminus\{0\}\\ \mbox{\tiny rad}(m^{\ast})|6hd}} \frac{1}{(m^{\ast})^2} \cdot
\sum\limits_{\substack{\tilde{m} \in \mathbb{Z}\setminus\{0\}\\ (\tilde{m},6hd)=1}} \frac{1}{\tilde{m}^2} \cdot 
\sum\limits_{u\in \mathbb{Z}}   \sum\limits_{z=1}^{d} \\ & \sum\limits_{\substack{ l_1 \le \xi/(k_2k_3hd) \\ (l_1,6hdm^{\ast}\tilde{m})=1}} \frac{\mu(l_1)}{l_1} \cdot \hat\Gamma\left(\frac{X^6m^{\ast}\tilde{m}}{k_2hdl_1^2}\right) \cdot   \hat\Psi\left(\frac{l_1^2u}{X^4k_3^2|m^{\ast}\tilde{m}|^2}\right) \times\\ & 
E(a^3b^4h^2m^{\ast}\tilde{m}\overline{l_1}^2,a^3b^2hz\tilde{m}^2\overline{l_1}^2;d) \cdot e\left(\frac{\overline{l_1}^2k_2^2\tilde{m}^4z^3+uz}{\tilde{q}}\right) \times\\ &
F\left(\overline{l_1}^2k_2^2\tilde{m}^4d^2,3\overline{l_1}^2k_2^2\tilde{m}^4dz,3\overline{l_1}^2k_2^2\tilde{m}^4z^2+u;\tilde{q}\right)
\cdot F(\overline{l_1}^2k_2^2q^2,0,u;\tilde{m}^2).
\end{split}
\end{equation}

Next, upon recalling \eqref{Idef} and \eqref{omegadef}, we split the function $\hat\Psi$ (with $s=l_1$) into
$$
\hat\Psi(\alpha)=\Omega(\alpha,\beta)+G(\alpha,\beta)
$$
and accordingly $E^-(X)$ into
\begin{equation} \label{E-split}
E^-(X)=E_{\Omega}(X)+E_{G}(X).
\end{equation}
Further, we cut summations, at the cost of errors of size $O(1)$, taking into account that $\hat\Gamma$, $\Omega$ and $G$ have rapid decay, which, in the case of $\Omega$, follows from the last case in \eqref{Omega}, and in the case of $G$ follows from \eqref{Gest}.   We set
$$
M:=\frac{\xi^2}{hdX^{6-\varepsilon}}, \quad U:=\frac{\xi^2}{X^{8-2\varepsilon}}, \quad L:=\max\left\{\left(\frac{|m^{\ast}\tilde{m}|}{k_2hd}\right)^{1/2}X^{3-\varepsilon}, 
\frac{|u|^{1/2}X^{4-\varepsilon}}{k_2k_3hd}\right\}, \quad 
\tilde{L}:=\frac{\xi}{k_2k_3hd}.
$$
Then
\begin{equation*} 
\begin{split}
E_{f}(X) \ll & 1+ X^{6} \sum\limits_{k_2|2} \sum\limits_{k_3|3} \mathop{\sum\limits_{\substack{h\le \xi/(k_2k_3) \\ (h,6)=1}} \sum\limits_{\substack{d\le \xi/(k_2k_3h)\\ (d,6h)=1}}}_{hd\le \xi^{\kappa}} \frac {\mu(k_2)}{k_2}\cdot \frac{\mu(k_3)}{k_3^3}\cdot \frac{\mu(h)}{h^2}\cdot \frac{\mu(d)}{d^3} \times\\ & 
\mathop{\sum\limits_{\substack{m^{\ast}\in \mathbb{Z}\setminus\{0\}\\ \mbox{\tiny rad}(m^{\ast})|6hd}}
\sum\limits_{\substack{\tilde{m} \in \mathbb{Z}\setminus\{0\}\\ (\tilde{m},6hd)=1}}}_{|m^{\ast}\tilde{m}|\le M} \frac{1}{\left(m^{\ast}\tilde{m}\right)^2} \cdot 
\sum\limits_{|u|\le U}   \sum\limits_{z=1}^{d} \\ & \sum\limits_{\substack{L< l_1 \le \tilde{L} \\ (l_1,6hdm^{\ast}\tilde{m})=1}} \frac{\mu(l_1)}{l_1} \cdot \hat\Gamma\left(\frac{X^6m^{\ast}\tilde{m}}{k_2hdl_1^2}\right) \cdot   f\left(\frac{l_1^2u}{X^4k_3^2|m^{\ast}\tilde{m}|^2},\frac{k_2^2h^2d^2l_1^4}{X^{12}|m^{\ast}\tilde{m}|^2}\right) \times\\ & 
E(a^3b^4h^2m^{\ast}\tilde{m}\overline{l_1}^2,a^3b^2hz\tilde{m}^2\overline{l_1}^2;d) \cdot e\left(\frac{\overline{l_1}^2k_2^2\tilde{m}^4z^3+uz}{\tilde{q}}\right) \times\\ &
F\left(\overline{l_1}^2k_2^2\tilde{m}^4d^2,3\overline{l_1}^2k_2^2\tilde{m}^4dz,3\overline{l_1}^2k_2^2\tilde{m}^4z^2+u;\tilde{q}\right)
\cdot F(\overline{l_1}^2k_2^2q^2,0,u;\tilde{m}^2),
\end{split}
\end{equation*}
for $f=\Omega,G$.
Breaking the summation over $l_1$ into residue classes modulo 
$$
q_1=[d,\tilde{q}]=\frac{d\tilde{q}}{(d,\tilde{q})}=\frac{q}{(d,|m^{\ast}|)}
$$
(recall that $q=\tilde{q}d$ and $q=k_3^3hd(m^{\ast})^2$), we get
\begin{equation} \label{Enew2}
\begin{split}
E_f(X) = & 1+ X^{6} \sum\limits_{k_2|2} \sum\limits_{k_3|3} \mathop{\sum\limits_{\substack{h\le \xi/(k_2k_3) \\ (h,6)=1}} \sum\limits_{\substack{d\le \xi/(k_2k_3h)\\ (d,6h)=1}}}_{hd\le \xi^{\kappa}} \frac {\mu(k_2)}{k_2}\cdot \frac{\mu(k_3)}{k_3^3}\cdot \frac{\mu(h)}{h^2}\cdot \frac{\mu(d)}{d^3} \times\\ & 
\mathop{\sum\limits_{\substack{m^{\ast}\in \mathbb{Z}\setminus\{0\}\\ \mbox{\tiny rad}(m^{\ast})|6hd}}
\sum\limits_{\substack{\tilde{m} \in \mathbb{Z}\setminus\{0\}\\ (\tilde{m},6hd)=1}}}_{|m^{\ast}\tilde{m}|\le M} \frac{1}{\left(m^{\ast}\tilde{m}\right)^2} \cdot 
\sum\limits_{|u|\le U}   \sum\limits_{\substack{j=1\\ (j,q_1)=1}}^{q_1}\\ & \sum\limits_{z=1}^{d}  E(a^3b^4h^2m^{\ast}\tilde{m}\overline{j}^2,a^3b^2hz\tilde{m}^2\overline{j}^2;d) \cdot
e\left(\frac{\overline{j}^2k_2^2\tilde{m}^4z^3+uz}{\tilde{q}}\right) \times\\ &
F\left(\overline{j}^2k_2^2\tilde{m}^4d^2,3\overline{j}^2k_2^2\tilde{m}^4dz,3\overline{j}^2k_2^2\tilde{m}^4z^2+u;\tilde{q}\right)  \times\\ &
\sum\limits_{\substack{L< l_1 \le \tilde{L} \\ (l_1,6\tilde{m})=1\\ l_1\equiv j \bmod{q_1}}} \frac{\mu(l_1)}{l_1} \cdot \hat\Gamma\left(\frac{X^6m^{\ast}\tilde{m}}{k_2hdl_1^2}\right) \times\\ &   f\left(\frac{l_1^2u}{X^4k_3^2|m^{\ast}\tilde{m}|^2},\frac{k_2^2h^2d^2l_1^4}{X^{12}|m^{\ast}\tilde{m}|^2}\right) 
\cdot F(\overline{l_1}^2k_2^2q^2,0,u;\tilde{m}^2).
\end{split}
\end{equation}

Now we bound the contribution $R_f(X)$ of $u=0$ to the right-hand side of \eqref{Enew2}. By definition of $G$, we have $R_{G}(X)=0$, and thus we are left with bounding 
$R_{\Omega}(X)$. Using \eqref{EF0} and \eqref{Fev} with $j$ in place of $l_1$, $\hat\Gamma(z)=O(1)$ and $\Omega(0)=O(1)$, which latter is the bound in the first case in \eqref{Omega}, and estimating the sum over $l_1$ trivially using $|\mu(l_1)|\le 1$, we get
\begin{equation} \label{ROmegaest}
R_{\Omega}(X) \ll X^{6+\varepsilon} \mathop{\sum\limits \sum\limits}_{hd\le \xi^{\kappa}} \frac{1}{(hd)^{5/4}}
\mathop{\sum\limits_{\substack{m^{\ast}\in \mathbb{Z}\setminus\{0\}\\ \mbox{\tiny rad}(m^{\ast})|6hd}}
\sum\limits_{\substack{\tilde{m} \in \mathbb{Z}\setminus\{0\}\\ (\tilde{m},6hd)=1}}}_{|m^{\ast}\tilde{m}|\le M} \frac{(d,|m^{\ast}|)^{3/4}}{|m^{\ast}|^{1/2}|\tilde{m}|}  \sum\limits_{\substack{x_1=1\\ x_1^2\equiv 0 \bmod{|\tilde{m}|}}}^{|\tilde{m}|} 1.
\end{equation}
The inner-most sum can be evaluated explicitly. If $s(\tilde{m})$ is the largest square dividing $\tilde{m}$, then
$$
\sum\limits_{\substack{x_1=1\\ x_1^2\equiv 0 \bmod{|\tilde{m}|}}}^{|\tilde{m}|} 1 = \sqrt{s({\tilde{m}})}.
$$
Now the sum over $\tilde{m}$ in \eqref{ROmegaest} can be estimated by
$$
\sum\limits_{\substack{\tilde{m} \in \mathbb{Z}\setminus\{0\}\\ (\tilde{m},6hd)=1\\|\tilde{m}|\le M/|m^{\ast}|}} \frac{\sqrt{s({\tilde{m}})}}{|\tilde{m}|}  \ll
\sum\limits_{1\le r\le M/|m^{\ast}|} \sqrt{r^2} \sum\limits_{1\le s\le M/(r^2|m^{\ast}|)} \frac{1}{r^2s} \ll X^{\varepsilon}.
$$
Further,
$$
\sum\limits_{\substack{m^{\ast}\in \mathbb{Z}\setminus \{0\}\\ |m^{\ast}|\le M\\ \mbox{\tiny rad}(m^{\ast})|6hd}} \frac{(d,|m^{\ast}|)^{3/4}}{|m^{\ast}|^{1/2}} \ll d^{1/4} \sum\limits_{\substack{m^{\ast}\in \mathbb{Z}\setminus \{0\}\\ |m^{\ast}|\le M\\ \mbox{\tiny rad}(m^{\ast})|6hd}} 1 \ll d^{1/4}X^{\varepsilon}.
$$
Hence,
$$
R_{\Omega}(X) \ll X^{6+\varepsilon} \mathop{\sum\limits \sum\limits}_{hd\le \xi^{\kappa}} \frac{1}{h^{5/4}d} \ll X^{6+\varepsilon}. 
$$
By a short calculation, it follows that $R(x)\ll X^{6+\varepsilon}$ and hence, we deduce from \eqref{EF}, \eqref{Fev} and \eqref{Enew2} that
\begin{equation} \label{Enew3}
\begin{split}
E_f(X) \ll & X^{6+\varepsilon} + X^{6+\varepsilon} \sum\limits_{k_2|2} \sum\limits_{k_3|3} \mathop{\sum\limits\sum\limits}_{hd\le \xi^{\kappa}} \frac{1}{(hd)^{3/2}} \cdot 
\mathop{\sum\limits_{\substack{m^{\ast}\in \mathbb{Z}\setminus\{0\}\\ \mbox{\tiny rad}(m^{\ast})|6hd}}
\sum\limits_{\substack{\tilde{m} \in \mathbb{Z}\setminus\{0\}\\ (\tilde{m},6hd)=1}}}_{|m^{\ast}\tilde{m}|\le M} \frac{(d,|m^{\ast}|)^{1/2}}{|m^{\ast}\tilde{m}|} \times\\ &
\sum\limits_{0<|u|\le U}   \left(u,hd|m^{\ast}|^2\right)^{1/2} \sum\limits_{\substack{x_1=1\\ x_1^2\equiv -\overline{3}u \bmod{|\tilde{m}|}}}^{|\tilde{m}|} \sum\limits_{\substack{j=1\\ (j,q_1)=1}}^{q_1}
 \left|Z_f(j) \right|,
\end{split}
\end{equation}
for $f=\Omega,G$, where 
\begin{equation} \label{Zdef}
\begin{split}
Z_f(j) := &
\sum\limits_{\substack{L< l_1 \le \tilde{L} \\ (l_1,6\tilde{m})=1\\ l_1\equiv j \bmod{q_1}}} \frac{\mu(l_1)}{l_1} \cdot \hat\Gamma\left(\frac{X^6m^{\ast}\tilde{m}}{k_2hdl_1^2}\right) \cdot   f\left(\frac{l_1^2u}{X^4k_3^2|m^{\ast}\tilde{m}|^2},\frac{k_2^2h^2d^2l_1^4}{X^{12}|m^{\ast}\tilde{m}|^2}\right) 
\cdot e\left(\gamma l_1\right),
\end{split}
\end{equation}
with
$$
\gamma:=\frac{(x_1^3+ux_1)\overline{k_2q}}{\tilde{m}^2}.
$$

\section{Partial summation over $l_1$}
In this section, we transform the inner-most sum in \eqref{Enew3}
$$
\sum\limits_{\substack{j=1\\ (j,q_1)=1}}^{q_1}
 \left|Z_f(j) \right|
 $$
by applying partial summation over $l_1$ to $Z_f(j)$. We shall assume that the variables\\ $k_2,k_3,h,d,m^{\ast},\tilde{m},u,x_1$ satisfy the summation conditions in \eqref{Enew}. In particular, \eqref{holds} holds, and we are in the case $0<|u|\le K$ in \eqref{Omega} and \eqref{Omegaprime}.  

We start with the case $f=G$. In this case, we have, by \eqref{G} with $s=l_1$,
\begin{equation} \label{ZGsplit}
Z_G(j)=Z_{G,1}(j)+Z_{G,-1}(j),
\end{equation}
where 
\begin{equation} \label{ZGomega}
\begin{split}
Z_{G,\omega}(j):= & \delta(u) \cdot e\left(\frac{\omega}{8}\right)\cdot \frac{X^4k_3^{1/2}|m^{\ast}\tilde{m}|}{(18k_2hd)^{1/2}|u|^{1/4}}\cdot \sum\limits_{\substack{L< l_1 \le \tilde{L} \\ (l_1,6\tilde{m})=1\\ l_1\equiv j \bmod{q_1}}} \frac{\mu(l_1)}{l_1^{5/2}} \cdot \hat\Gamma\left(\frac{X^6m^{\ast}\tilde{m}}{k_2hdl_1^2}\right) \times\\ &  \hat\Gamma\left(\omega\cdot \frac{X^4|u|^{1/2}}{3^{1/2}k_2k_3hdl_1}\right)  \cdot 
e\left(\gamma_\omega l_1\right)
\end{split}
\end{equation}
with 
$$
\gamma_{\omega}:=\gamma-\omega\cdot \frac{2|u|^{3/2}}{3^{3/2}k_3^3k_2hd|m^{\ast}\tilde{m}|^2}.
$$

Next, we write 
\begin{equation} \label{Sswdef}
\mathcal{S}_j(s,w):= \sum\limits_{\substack{L< n \le s \\ (n,6\tilde{m})=1\\ n \equiv j \bmod{q_1}}} \mu(n)\cdot e\left(w n\right)
\end{equation}
and remove the weight functions on the right-hand side of \eqref{ZGomega} using partial summation, leading to
\begin{equation*} 
\begin{split}
Z_{G,\omega}(j)= & \delta(u) \cdot e\left(\frac{\omega}{8}\right)\cdot \frac{X^4k_3^{1/2}|m^{\ast}\tilde{m}|}{(18k_2hd)^{1/2}|u|^{1/4}} \times\\ & 
\left(\frac{1}{\tilde{L}^{5/2}} \cdot \hat\Gamma\left(\frac{X^6m^{\ast}\tilde{m}}{k_2hd\tilde{L}^2}\right) \cdot 
\hat\Gamma\left(\omega\cdot \frac{X^4|u|^{1/2}}{3^{1/2}k_2k_3hd\tilde{L}}\right) \cdot \mathcal{S}_j(\tilde{L},\gamma_{\omega})+\right. \\ &  
\frac{5}{2} \cdot \int\limits_{L}^{\tilde{L}} \frac{1}{s^{7/2}} \cdot \hat\Gamma\left(\frac{X^6m^{\ast}\tilde{m}}{k_2hds^2}\right) \cdot 
\hat\Gamma\left(\omega\cdot \frac{X^4|u|^{1/2}}{3^{1/2}k_2k_3hds}\right) \cdot \mathcal{S}_j(s,\gamma_{\omega}) \ ds + \\
&  
\frac{2X^6m^{\ast}\tilde{m}}{k_2hd} \cdot \int\limits_{L}^{\tilde{L}} \frac{1}{s^{11/2}} \cdot \hat\Gamma'\left(\frac{X^6m^{\ast}\tilde{m}}{k_2hds^2}\right) \cdot 
\hat\Gamma\left(\omega\cdot \frac{X^4|u|^{1/2}}{3^{1/2}k_2k_3hds}\right) \cdot \mathcal{S}_j(s,\gamma_{\omega}) \ ds + \\ &  \left.
\frac{\omega X^4|u|^{1/2}}{3^{1/2}k_2k_3hd} \cdot \int\limits_{L}^{\tilde{L}} \frac{1}{s^{9/2}} \cdot \hat\Gamma\left(\frac{X^6m^{\ast}\tilde{m}}{k_2hds^2}\right) \cdot 
\hat\Gamma'\left(\omega\cdot \frac{X^4|u|^{1/2}}{3^{1/2}k_2k_3hds}\right) \cdot \mathcal{S}_j(s,\gamma_{\omega}) \ ds \right).
\end{split}
\end{equation*}
Hence, using $\hat\Gamma(z),\hat\Gamma'(z)=O(1)$, it follows that
\begin{equation} \label{ZGomega1}
\begin{split}
& \sum\limits_{\substack{j=1\\ (j,q_1)=1}}^{q_1} |Z_{G,\omega}(j)|\ll
\frac{X^4|m^{\ast}\tilde{m}|}{(hd)^{1/2}|u|^{1/4}\tilde{L}^{5/2}} \cdot 
\sum\limits_{\substack{j=1\\ (j,q_1)=1}}^{q_1} \left|\mathcal{S}_j(\tilde{L},\gamma_{\omega})\right|+ \\ &  
\int\limits_{L}^{\tilde{L}} \left(\frac{X^4|m^{\ast}\tilde{m}|}{(hd)^{1/2}|u|^{1/4}s^{7/2}} + \frac{X^{10}|m^{\ast}\tilde{m}|^2}{(hd)^{3/2}|u|^{1/4}s^{11/2}} + 
\frac{X^8|u|^{1/4}|m^{\ast}\tilde{m}|}{(hd)^{3/2}s^{9/2}} \right) \cdot  \sum\limits_{\substack{j=1\\ (j,q_1)=1}}^{q_1} \left|\mathcal{S}_j(s,\gamma_{\omega})\right|\ ds.
\end{split}
\end{equation}

Now we turn to the case $f=\Omega$. Then using \eqref{Zdef} and partial summation, we get
\begin{equation*} 
\begin{split}
Z_{\Omega}(j)= & 
\left.\frac{1}{\tilde{L}} \cdot \hat\Gamma\left(\frac{X^6m^{\ast}\tilde{m}}{k_2hd\tilde{L}^2}\right) \cdot  
\Omega\left(\frac{u\tilde{L}^2}{X^4k_3^2|m^{\ast}\tilde{m}|^2},\frac{k_2^2h^2d^2\tilde{L}^4}{X^{12}|m^{\ast}\tilde{m}|^2}\right) \mathcal{S}_j(\tilde{L},\gamma)+\right. \\ &  
 \int\limits_{L}^{\tilde{L}} \frac{1}{s^2} \cdot \hat\Gamma\left(\frac{X^6m^{\ast}\tilde{m}}{k_2hds^2}\right) \cdot 
\Omega\left(\frac{us^2}{X^4k_3^2|m^{\ast}\tilde{m}|^2},\frac{k_2^2h^2d^2s^4}{X^{12}|m^{\ast}\tilde{m}|^2}\right)  \cdot \mathcal{S}_j(s,\gamma) \ ds + \\
&  
\frac{2X^6m^{\ast}\tilde{m}}{k_2hd} \cdot \int\limits_{L}^{\tilde{L}} \frac{1}{s^4} \cdot \hat\Gamma'\left(\frac{X^6m^{\ast}\tilde{m}}{k_2hds^2}\right) \cdot 
\Omega\left(\frac{us^2}{X^4k_3^2|m^{\ast}\tilde{m}|^2},\frac{k_2^2h^2d^2s^4}{X^{12}|m^{\ast}\tilde{m}|^2}\right)  \cdot \mathcal{S}_j(s,\gamma) \ ds - \\ &  \left.
\int\limits_{L}^{\tilde{L}}  \frac{1}{s}\cdot \hat\Gamma\left(\frac{X^6m^{\ast}\tilde{m}}{k_2hds^2}\right) \cdot 
\frac{\partial}{\partial s}\Omega\left(\frac{us^2}{X^4k_3^2|m^{\ast}\tilde{m}|^2},\frac{k_2^2h^2d^2s^4}{X^{12}|m^{\ast}\tilde{m}|^2}\right)  \cdot \mathcal{S}_j(s,\gamma) \ ds
\right..
\end{split}
\end{equation*}
Hence, using  $\hat\Gamma(z)=O(1)$ and the bounds for $\Omega$ and $\Omega'$ in the case $0<|u|\le K$ in \eqref{Omega} and \eqref{Omegaprime}, it follows that
\begin{equation} \label{ZGomega2}
\begin{split}
& \sum\limits_{\substack{j=1\\ (j,q_1)=1}}^{q_1} |Z_{\Omega}(j)|\ll
X^{\varepsilon} \cdot \frac{X^4|m^{\ast}\tilde{m}|^2}{|u|\tilde{L}^{3}}\cdot  \sum\limits_{\substack{j=1\\ (j,q_1)=1}}^{q_1} 
\left|\mathcal{S}_j(\tilde{L},\gamma)\right|+ X^{\varepsilon}\times \\ &
\int\limits_{L}^{\tilde{L}} \left(\frac{X^4|m^{\ast}\tilde{m}|^2}{|u|s^{4}}+ \frac{X^{10}|m^{\ast}\tilde{m}|^3}{hd|u|s^6} + 
\frac{1}{s^2}+\frac{X^4|m^{\ast}\tilde{m}|}{(hd)^{1/2}|u|^{1/4}s^{7/2}}
 +\frac{X^8|u|^{1/4}|m^{\ast}\tilde{m}|}{(hd)^{3/2}s^{9/2}} \right) \times\\ & 
\sum\limits_{\substack{j=1\\ (j,q_1)=1}}^{q_1} \left|\mathcal{S}_j(s,\gamma)\right|\ ds.
\end{split}
\end{equation}

\section{Averaging over $j$ and $l_1$}
The next step is to estimate
$$
\sum\limits_{\substack{j=1\\ (j,q_1)=1}}^{q_1} \left|\mathcal{S}_j(s,w)\right|
$$
for any $w\in \mathbb{R}$, where $\mathcal{S}_j(s,w)$ is defined as in \eqref{Sswdef}.
Let $R\le X^{100}$ be a positive integer, to be specified later. Using Dirichlet's approximation theorem, there exist an integer $a$ and a positive integer $r$ such that $(a,r)=1$,
$r\le R$ and $w= a/r+\beta$ with $|\beta|\le 1/(rR)$. Using partial summation, it follows that
\begin{equation} \label{this}
\begin{split}
&  \mathcal{S}_j(s,w)\\ = & e(\beta s) \sum\limits_{\substack{L< n \le s \\ (n,6\tilde{m})=1\\ n \equiv j \bmod{q}_1}} \mu(n)\cdot e\left(\frac{a}{r}\cdot n\right) -
2\pi i\beta \int\limits_{L}^s  e(\beta t) \cdot \left(\sum\limits_{\substack{L< n \le t \\ (n,6\tilde{m})=1\\ n \equiv j \bmod{q_1}}} \mu(n)\cdot e\left(\frac{a}{r}\cdot n\right)\right) dt \\
\ll & \left| \sum\limits_{\substack{L< n \le s \\ (n,6\tilde{m})=1\\ n \equiv j \bmod{q_1}}} \mu(n)\cdot e\left(\frac{a}{r}\cdot n\right)\right| +
\frac{1}{rR} \int\limits_{L}^s  \left|\sum\limits_{\substack{L< n \le t \\ (n,6\tilde{m})=1\\ n \equiv j \bmod{q_1}}} \mu(n)\cdot e\left(\frac{a}{r}\cdot n\right)\right|\ dt. 
\end{split}
\end{equation}

Now we write $f:=(n,r)$, $n_1:=n/f$ and $r_1:=r/f$. Then 
$$
 e\left(\frac{a}{r}\cdot n\right)=e\left(\frac{a}{r_1}\cdot n_1\right)= \frac{1}{\varphi(r_1)} \cdot \sum\limits_{\chi_1 \bmod{r_1}} \overline{\chi}(an_1)\tau(\chi_1),
$$
where 
$$
\tau(\chi_1):=\sum\limits_{x=1}^{r_1} \chi_1(x)\cdot e\left(\frac{x}{r_1}\right)
$$
is the Gauss sum for the Dirichlet character $\chi_1$. Using this and
$$
\mu(n)=\begin{cases} \mu(n_1)\mu(f) & \mbox{ if } (n_1,f)=1,\\ 0 & \mbox{ if } (n_1,f)>1, \end{cases}  
$$
and detecting the coprimality condition $(n_1,6\tilde{m}f)=1$ using the principal character $\chi_0$ modulo $6\tilde{m}f$ and the congruence condition $n_1f\equiv j \bmod{q_1}$ by Dirichlet characters $\chi$ modulo $q_1$ (recall that $(j,q)=1$), we arrive at
\begin{equation*}
\begin{split}
\sum\limits_{\substack{L< n \le t \\ (n,6\tilde{m})=1\\ n \equiv j \bmod{q_1}}} \mu(n)\cdot e\left(\frac{a}{r}\cdot n\right)= & \frac{1}{\varphi(q_1)}\cdot  \sum\limits_{\chi \bmod{q_1}} \overline{\chi}(j)
\sum\limits_{f|r} \mu(f) \chi(f) \cdot \frac{1}{\varphi(r_1)} \cdot \sum\limits_{\chi_1 \bmod{r_1}} \overline{\chi}_1(a) \chi_1(f) \tau(\chi_1) \times\\ 
& \sum\limits_{L/f<n_1\le t/f} \chi\chi_0\chi_1(n_1)\mu(n_1).
\end{split}
\end{equation*}
We shall also obtain a saving by averaging over $j \bmod q$, where $(j,q)=1$. Using the orthogonality relations for Dirichlet characters, the bound $|\tau(\chi_1)|\le \sqrt{r_1}$,
and the Riemann Hypothesis for Dirichlet $L$-functions which implies that
$$
\sum\limits_{L/f<n_1\le t/f} \chi\chi_0\chi_1(n_1)\mu(n_1) \ll \left(\frac{t}{f}\right)^{1/2} X^{\varepsilon},
$$
we deduce that
\begin{equation*}
\begin{split}
& \sum\limits_{\substack{j=1\\ (j,q_1)=1}}^{q_1} \left|\sum\limits_{\substack{L< n \le t \\ (n,6\tilde{m})=1\\ n \equiv j \bmod{q_1}}} \mu(n)\cdot e\left(\frac{a}{r}\cdot n\right)\right|^2 = 
\frac{1}{\varphi(q_1)} \sum\limits_{\chi \bmod{q_1}} \left|
\sum\limits_{f|r} \mu(f) \chi(f) \cdot \frac{1}{\varphi(r_1)} \times \right. \\ &  \left. \sum\limits_{\chi_1 \bmod{r_1}} \overline{\chi}_1(a) \chi_1(f) \tau(\chi_1) \cdot
\sum\limits_{L/f<n_1\le t/f} \chi\chi_0\chi_1(n_1)\mu(n_1)\right|^2
\ll rtX^{\varepsilon},
\end{split}
\end{equation*}
and therefore, by Cauchy-Schwarz,
\begin{equation*}
\sum\limits_{\substack{j=1\\ (j,q_1)=1}}^{q_1} \left|\sum\limits_{\substack{L< n \le t \\ (n,6\tilde{m})=1\\ n \equiv j \bmod{q_1}}} \mu(n)\cdot e\left(\frac{a}{r}\cdot n\right)\right|\ll (q_1rt)^{1/2}X^{\varepsilon}.
\end{equation*}
Using this together with \eqref{this}, we get
\begin{equation*} 
\sum\limits_{\substack{j=1\\ (j,q_1)=1}}^{q_1} |\mathcal{S}_j(s,w)|\ll q_1^{1/2}\left(r^{1/2}s^{1/2}+\frac{s^{3/2}}{r^{1/2}R}\right) \le q_1^{1/2}
\left(R^{1/2}s^{1/2}+\frac{s^{3/2}}{R}\right).
\end{equation*}
Now fixing $R:= s^{2/3}$, it follows that
\begin{equation} \label{plug}  
\sum\limits_{\substack{j=1\\ (j,q_1)=1}}^{q_1} |\mathcal{S}_j(s,w)|\ll q_1^{1/2}s^{5/6}X^{\varepsilon}.
\end{equation}

\section{Proof of Theorem \ref{density2}}
To prove Theorem \ref{density2}, it remains to bound the error term $E^{-}(X)$. 
Combing \eqref{ZGsplit}, \eqref{ZGomega1}, \eqref{ZGomega2} and \eqref{plug}, we obtain
\begin{equation} \label{wutz}
\begin{split}
& \sum\limits_{\substack{j=1\\ (j,q_1)=1}}^{q_1} \left(\left|Z_G(j)\right|+\left|Z_\Omega(j)\right|\right)\ll  q_1^{1/2}X^{\varepsilon}\left(\frac{X^4|m^{\ast}\tilde{m}|}{(hd)^{1/2}|u|^{1/4}L^{5/3}}+\frac{X^{10}|m^{\ast}\tilde{m}|^2}{(hd)^{3/2}|u|^{1/4}L^{11/3}}+\right. \\ & \left.
\frac{X^8|u|^{1/4}|m^{\ast}\tilde{m}|}{(hd)^{3/2}L^{8/3}}+
\frac{X^4|m^{\ast}\tilde{m}|^2}{|u|L^{13/6}}+\frac{X^{10}|m^{\ast}\tilde{m}|^3}{hd|u|L^{25/6}}+
\frac{1}{L^{1/6}}\right).
\end{split}
\end{equation}
Recalling
\begin{equation} \label{recall}
q_1=\frac{q}{(d,|m^{\ast}|)}=\frac{k_3^3hd(m^{\ast})^2}{(d,|m^{\ast}|)}  \quad \mbox{and} \quad L:=\max\left\{\left(\frac{|m^{\ast}\tilde{m}|}{k_2hd}\right)^{1/2}X^{3-\varepsilon}, 
\frac{|u|^{1/2}X^{4-\varepsilon}}{k_2k_3hd}\right\},
\end{equation}
we deduce that
\begin{equation} \label{combining}
\sum\limits_{\substack{j=1\\ (j,q_1)=1}}^{q_1} \left(\left|Z_G(j)\right|+\left|Z_\Omega(j)\right|\right)\ll X^{\varepsilon}\cdot \frac{|m^{\ast}|}{(d,|m^{\ast}|)^{1/2}}\cdot 
\sum\limits_{i=1}^{4} u^{\alpha_i}\left|m^{\ast}\tilde{m}\right|^{\beta_i}(hd)^{\gamma_i}
X^{\delta_i},
\end{equation}
where 
\begin{equation} \label{abcd}
\begin{split}
(\alpha_1,\beta_1,\gamma_1,\delta_1):= &\left(-\frac{1}{4},\frac{1}{6},\frac{5}{6},-1\right),\\
(\alpha_2,\beta_2,\gamma_2,\delta_2):= &\left(-\frac{13}{12},1,\frac{2}{3},-\frac{8}{3}\right),\\
(\alpha_3,\beta_3,\gamma_3,\delta_3):= &\left(-1,\frac{11}{12},\frac{19}{12},-\frac{5}{2}\right),\\
(\alpha_4,\beta_4,\gamma_4,\delta_4):= &\left(-\frac{1}{12},0,\frac{2}{3},-\frac{2}{3}\right).
\end{split}
\end{equation}
Here we use the second term in the maximum in \eqref{recall} to bound the third term and sixth term in the sum on the right-hand side of \eqref{wutz} and the first term in the said maximum to bound all the other terms in the said sum. Combining \eqref{E-split}, \eqref{Enew3} and \eqref{combining}, it follows that
\begin{equation} \label{finalE}
E^-(X)\ll  X^{6+\varepsilon}+\sum\limits_{i=1}^{4}A(\alpha_i,\beta_i,\gamma_i,\delta_i),
\end{equation}
where
\begin{equation} \label{Adef} 
\begin{split}
A(\alpha,\beta,\gamma,\delta):= & X^{6+\varepsilon}  \mathop{\sum\sum}_{(hd)\le \xi^{\kappa}} 
\frac{1}{(hd)^{3/2}} \cdot 
\mathop{\sum\limits_{\substack{m^{\ast}\in \mathbb{Z}\setminus\{0\}\\ \mbox{\tiny rad}(m^{\ast})|6hd}}
\sum\limits_{\substack{\tilde{m} \in \mathbb{Z}\setminus\{0\}\\ (\tilde{m},6hd)=1}}}_{|m^{\ast}\tilde{m}|\le M}  \frac{1}{|m^{\ast}\tilde{m}|} \times\\ &
\sum\limits_{0<|u|\le U}   \left(u,hd|m^{\ast}|^2|\tilde{m}|\right)^{1/2} \sum\limits_{\substack{x_1=1\\ x_1^2\equiv -\overline{3}u \bmod{|\tilde{m}|}}}^{|\tilde{m}|} |u|^{\alpha}|m^{\ast}\tilde{m}|^{\beta}(hd)^{\gamma}X^{\delta}.
 \end{split}
\end{equation}

Now to estimate $E^-(X)$, we bound $A(\alpha,\beta,\gamma,\delta)$ for general parameters. We first  
examine the cardinality
$$
\sum\limits_{\substack{x_1=1\\ x_1^2\equiv -\overline{3}u \bmod{|\tilde{m}|}}}^{|\tilde{m}|} 1 =\sharp\{ x\in \{1,...,\tilde{m}\} : x_1^2\equiv -\overline{3}u \bmod{|\tilde{m}|}\},
$$
where we recall that $(|\tilde{m}|,6)=1$ and $\overline{3}$ is a multiplicative inverse of $3$ modulo $|\tilde{m}|$. We observe that a necessary condition for  the above congruence to be solvable is that $(u,|\tilde{m}|)$ is a perfect square. In this case, let $(u,|\tilde{m}|)=v^2$, $v>0$. Then $v|x_1$, and the congruence reduces
to $x_2^2\equiv \overline{3}u_2 \bmod{m_2}$, where $x_2=x_1/v$, $u_2=u/v^2$ and $m_2=|\tilde{m}|/v^2$. Hence,
\begin{equation*}
\begin{split}
\sharp\{ x\in \{1,...,\tilde{m}\} : x_1^2\equiv -\overline{3}u \bmod{|\tilde{m}|}\}= & v^2\cdot \sharp\{ x\in \{1,...,m_2\} : x_2^2\equiv -\overline{3}u_2 \bmod{m_2}\}\\ 
\ll & v^2 \cdot m_2^{\varepsilon}
\ll (u,|\tilde{m}|) X^{\varepsilon}.
\end{split}
\end{equation*}
We further observe that
$$
\sum\limits_{0<|u|\le U} \left(u,hd|m^{\ast}|^2|\tilde{m}|\right)^{1/2}|u|^{\alpha} \ll U^{\tilde{\alpha}}(hd|m^{\ast}\tilde{m}|)^{\varepsilon}\ll U^{\tilde{\alpha}}X^{\varepsilon} \ll \frac{\xi^{2\tilde{\alpha}}}{X^{8\tilde{\alpha}}}\cdot X^{\varepsilon}
$$
with $\tilde\alpha:=\max(0,1+\alpha)$, where we recall $U=\xi^2X^{2\varepsilon-8}$. Combining the above gives
\begin{equation} \label{first}
\sum\limits_{0<|u|\le U} \sum\limits_{\substack{x_1=1\\ x_1^2\equiv -\overline{3}u \bmod{|\tilde{m}|}}}^{|\tilde{m}|} \left(u,hd|m^{\ast}|^2|\tilde{m}|\right) |u|^{\alpha}\ll \frac{\xi^{2\tilde{\alpha}}}{X^{8\tilde{\alpha}}}\cdot X^{\varepsilon}.
\end{equation}
Furthermore, we we have
\begin{equation} \label{second}
\mathop{\sum\limits_{\substack{m^{\ast}\in \mathbb{Z}\setminus\{0\}\\ \mbox{\tiny rad}(m^{\ast})|6hd}}
\sum\limits_{\substack{\tilde{m} \in \mathbb{Z}\setminus\{0\}\\ (\tilde{m},6hd)=1}}}_{|m^{\ast}\tilde{m}|\le M} \frac{|m^{\ast}|^{1+\beta}|\tilde{m}|^{\beta}}{|m^{\ast}\tilde{m}|} \ll 
M^{\tilde{\beta}+\varepsilon} \sum\limits_{\substack{0<|m^{\ast}|\le M\\ \mbox{\tiny rad}(m^{\ast})|6hd}} 1 \ll  M^{\tilde{\beta}}(Mhd)^{\varepsilon} \ll \frac{\xi^{2\tilde\beta}}{(hd)^{\tilde\beta}X^{6\tilde\beta}}\cdot X^{\varepsilon}
\end{equation}
with $\tilde\beta:=\max(0,\beta)$, where we recall $M=\xi^2/ (hdX^{6-\varepsilon})$.  Using \eqref{Adef}, \eqref{first} and \eqref{second}, we get
\begin{equation} \label{Aesti}
\begin{split}
A(\alpha,\beta,\gamma,\delta)\ll & \xi^{2\tilde{\alpha}+2\tilde{\beta}}X^{6+\delta-8\tilde{\alpha}-6\tilde{\beta}+\varepsilon}\sum\limits_{h\le \xi^{\kappa}} h^{\gamma-\tilde{\beta}-3/2} 
\sum\limits_{d\le \xi^{\kappa}/h} d^{\gamma-\tilde{\beta}-3/2}\\
\ll & \xi^{2\tilde{\alpha}+2\tilde{\beta}+\kappa\tilde{\gamma}}X^{6+\delta-8\tilde\alpha-6\tilde{\beta}+\varepsilon}
\end{split}
\end{equation}
with $\tilde{\gamma}:=\max\{0,\gamma-\tilde{\beta}-1/2\}$. 

Now using \eqref{abcd}, \eqref{finalE} and  \eqref{Aesti}, we compute that
\begin{equation} \label{finalE-}
E^-(X)\ll X^{\varepsilon}\left(X^{6} + \xi^{11/6+\kappa/6}X^{-2}+\xi^{2}X^{-8/3}\right).
\end{equation}
Combining \eqref{splitting}, \eqref{S2est}, \eqref{S2split}, \eqref{Msplit}, \eqref{themain}, \eqref{E0est}, \eqref{Esplit}, \eqref{E+bound} and \eqref{finalE-}, we arrive at
\begin{equation*} 
\begin{split}
S(X)= & X^{10} \cdot \frac{1}{3}\cdot  \prod\limits_{p>3}  \left(1-\frac{2p-1}{p^3}\right)+\\ & O\left(X^{\varepsilon}\left(X^6+\frac{X^{16}}{\xi^2}+\frac{X^{10}}{\xi}+\xi^{(1-\kappa)/2}X^4+\xi^{2-\kappa}+\frac{\xi^{11/6+\kappa/6}}{X^2}+\frac{\xi^{2}}{X^{8/3}}\right)\right).
\end{split}
\end{equation*}
Finally, taking
$$
\kappa:=\frac{16}{31} \quad \mbox{and} \quad \xi:=X^{124/27},
$$
we obtain the result in Theorem \ref{density2}.

\end{document}